\documentclass[12pt]{amsart}
\usepackage{amsmath,amssymb,hyperref,geometry,graphicx,tikz,ytableau,fp,mathdots,amsthm}
\usetikzlibrary{cd}

\newcounter{intro}[section]

\newtheorem{theorem}[equation]{Theorem}
\newtheorem{lemma}[equation]{Lemma}
\newtheorem{corollary}[equation]{Corollary}
\newtheorem{proposition}[equation]{Proposition}
\newtheorem{result}[intro]{Result}
\theoremstyle{definition}
\newtheorem{example}[equation]{Example}

\newtheorem{definition}[equation]{Definition}
\newtheorem{Definition}[intro]{Definition}

\newcommand{\R}{\mathbb{R}}
\newcommand{\Vol}{\mathrm{Vol}}
\newcommand{\D}{\mathcal{D}}
\newcommand{\Z}{\mathbb{Z}}
\newcommand{\LM}{\overline{\mathcal{LM}}}
\renewcommand{\P}{\mathbb{P}}
\newcommand{\C}{\mathbb{C}}

\newcommand{\T}{\mathbb{T}}

\newcommand{\bL}{\mathbb{L}}

\newcommand{\I}{\mathcal{I}}
\newcommand{\J}{\mathcal{J}}
\newcommand{\Conv}{\mathrm{Conv}}
\newcommand{\cL}{\mathcal{L}}

\newcommand{\rk}{\mathrm{rk}}
\renewcommand{\phi}{\varphi}
\newcommand{\F}{\mathcal{F}}
\newcommand{\M}{\overline{\mathcal{M}}}
\newcommand{\sM}{\mathsf{M}}
\newcommand{\cl}{\mathrm{cl}}
\newcommand{\ochi}{\overline{\chi}}

\numberwithin{equation}{section}

\begin{document}

\title{Matroid psi classes}
\author[J.~Dastidar]{Jeshu Dastidar}
\address{Department of Mathematics, University of California, Davis}
\email{jdastidar@ucdavis.edu }
\author[D.~Ross]{Dustin Ross}
\address{Department of Mathematics, San Francisco State University}
\email{rossd@sfsu.edu}

\begin{abstract}
Motivated by the intersection theory of moduli spaces of curves, we introduce psi classes in matroid Chow rings and prove a number of properties that naturally generalize properties of psi classes in Chow rings of Losev-Manin spaces. We use these properties of matroid psi classes to give new proofs of (1) a Chow-theoretic interpretation for the coefficients of the reduced characteristic polynomials of matroids, (2) explicit formulas for the volume polynomials of matroids, and (3) Poincar\'e duality for matroid Chow rings.
\end{abstract}

\maketitle

\section{Introduction}

Psi classes are special divisors that are ubiquitous in the study of the intersection theory of moduli spaces of curves. Psi classes arise naturally when computing products of boundary classes in $A^*(\M_{g,n})$ whose strata have excess intersection. In particular, any product of boundary classes can be written in terms of polynomials of psi classes on other boundary classes, and the top degrees of these polynomials are determined by the Witten--Kontsevich theorem \cite{Witten,Kontsevich}. In genus zero, this procedure takes on an especially simple form. Given $k$ distinct boundary divisors $D_1,\dots,D_k\in A^1(\M_{0,n})$, their product is also a (possibly empty) boundary class, and any monomial in these boundary divisors can be written as
\begin{equation}\label{boundarydivisorproduct}
D_1^{d_1}\cdots D_k^{d_k}=D_1\cdots D_k\prod_{i=1}^k(-\psi_{D_i}^--\psi_{D_i}^+)^{d_i-1}\in A^*(\M_{0,n}),
\end{equation}
where $\psi_{D}^{\pm}$ are certain psi classes associated to each divisor $D$. Moreover, if $\sum d_i=\dim(\M_{0,n})$, then the degree of the expression in the right-hand side of \eqref{boundarydivisorproduct} is a product of polynomials of psi classes on smaller dimensional moduli spaces, all of which are computable. The aim of this paper is to develop an analogue of these techniques in Chow rings of matroids. 

Matroids are combinatorial structures that generalize the behavior of finite sets of vectors, and Chow rings of matroids were introduced by Feichtner and Yuzvinsky \cite{FY}. In this work, we explore an appealing parallel between Chow rings of matroids and Chow rings of moduli spaces of curves. \emph{We introduce matroid psi classes in Chow rings of matroids and we show that they behave analogously to the usual psi classes in the Chow rings of moduli spaces of genus zero curves. As a first application, we then use psi classes to give simplified proofs of a number of recent foundational results concerning matroid Chow rings.} 

\subsection{Summary of results}

Given a loopless matroid $\sM=(E,\cL)$ consisting of a finite set $E$ and a lattice of flats $\cL\subseteq 2^E$, the Chow ring $A^*(\sM)$ is a graded $\Z$-algebra generated by matroid divisors $D_F\in A^1(\sM)$, one for each proper flat $F\in\cL^*=\cL\setminus\{\emptyset, E\}$ (see Subsection~\ref{sec:matroidbackground} for precise definitions). The primary objects of study in this paper are the following classes.

\begin{Definition}[Definition~\ref{def:matroidpsiclass}]
For any $F\in\cL$ and $e\in E$, define $\psi_F^{\pm}\in A^1(\sM)$ by
\[
\psi_F^-=\sum_{G\in\cL^*\atop e\in G}D_G-\sum_{G\in\cL^*\atop G\supseteq F}D_G\;\;\;\text{ and }\;\;\;\psi_F^+=\sum_{G\in\cL^*\atop e\notin G}D_G-\sum_{G\in\cL^*\atop G\subseteq F}D_G.
\]
\end{Definition}

The Chow classes $\psi_F^{\pm}$ do not depend on the choice of $e\in E$, which is why it is suppressed from the notation. As we will see, the definition of matroid psi classes is an immediate generalization of an expression for psi classes in terms of boundary divisors in Losev-Manin moduli spaces (see Lemma~\ref{lem:psilinearcombo}). 

After defining matroid psi classes, we establish the following analogue of Equation \eqref{boundarydivisorproduct}.

\begin{result}[Corollary~\ref{cor:selfintersectionmatroid}]
If $F_1,\dots,F_k$ are distinct flats of $\sM$ and $d_1,\dots,d_k$ are positive integers, then
\[
D_{F_1}^{d_1}\cdots D_{F_k}^{d_k}=D_{F_1}\cdots D_{F_k}\prod_{i=1}^k(-\psi_{F_i}^--\psi_{F_i}^+)^{d_i-1}\in A^*(\sM).
\]
\end{result}

This result allows us to express any monomial in matroid divisors as a squarefree expression along with a polynomial in psi classes. In the case that the product is in the top graded piece of the matroid Chow ring, our next result allows us to compute the degrees of the terms in Result B in terms of degrees of the special classes $\psi_0=\psi_{\emptyset}^+$ and $\psi_\infty=\psi_E^-$.

\begin{result}[Proposition~\ref{prop:degreeofpsi2}]
If $\emptyset=F_0\subsetneq F_1\subsetneq\dots\subsetneq F_{k}\subsetneq F_{k+1}=E$ are flats of $\sM$ and $a_0^+,a_1^-,a_1^+,\dots,a_{k}^-,a_k^+,a_{k+1}^-$ are nonnegative integers ,then
\[
\deg_\sM\Big(D_{F_1}\cdots D_{F_k}\prod_{i=0}^{k}(\psi_{F_{i}}^+)^{a_{i}^+}(\psi_{F_{i+1}}^-)^{a_{i+1}^-}\Big)=\prod_{i=0}^k\deg_{\sM[F_i,F_{i+1}]}\Big(\psi_0^{a_{i}^+}\psi_\infty^{a_{i+1}^-}\Big)
\]
\end{result}

In the above formula, $\sM[F_i,F_{i+1}]$ denotes the contraction by $F_{i}$ of the restriction of $\sM$ to $F_{i+1}$. In order to use Results B and C to explicitly compute degrees of polynomials in the generators, we use properties of psi classes to give a new proof of the following result, which had previously been proved by Huh and Katz \cite[Proposition~5.2]{HuhKatz}.

\begin{result}[Proposition~\ref{prop:degreeofpsi}]
If $\sM$ is a matroid and $a,b$ are nonnegative integers, then
\[
\deg_\sM(\psi_0^a\psi_\infty^b)=\begin{cases}
\mu^a(\sM) &\text{if }a+b=\rk(\sM)-1,\\
0 &\text{else,}
\end{cases}
\]
where $\mu^a(\sM)$ is the $a$th unsigned coefficient of the reduced characteristic polynomial of $\sM$.
\end{result}

Results B, C, and D provide an efficient algorithm for computing the degree of any monomial of matroid divisors. As a direct consequence of this algorithm, we recover a recent theorem of Eur \cite[Theorem 3.2]{Eur} that computes the coefficients of the volume polynomials of matroids.

\begin{result}[Theorem~\ref{thm:eur2}]
If $\emptyset=F_0\subsetneq F_1\subsetneq\dots\subsetneq F_{k}\subsetneq F_{k+1}=E$ are flats of $\sM$ and $d_1,\dots,d_k$  are positive integers that sum to $\rk(\sM)-1$, then
\[
\deg_\sM(D_{F_1}^{d_1}\cdots D_{F_k}^{d_k})=(-1)^{\rk(\sM)-k-1}\prod_{i=1}^k{d_i-1\choose \tilde d_i-\rk(F_i)}\mu^{\tilde d_i-\rk(F_i)}(\sM[F_i,F_{i+1}]),
\]
with
\[
\tilde d_j=\sum_{i=1}^j d_i.
\] 
\end{result}

Our developments can also be used to recover a recent theorem of Backman, Eur, and Simpson \cite[Theorem 5.2.4]{BES} that computes degrees of monomials in the ``simplicial'' generators, which, as it turns out, are nothing more than the psi classes $\psi_F^-$.

\begin{result}[Theorem~\ref{thm:BES}]
If $F_1,\dots,F_r$ are nonempty flats with $r=\rk(\sM)-1$, then
\[
\deg_\sM(\psi_{F_1}^-\dots\psi_{F_r}^-)=
\begin{cases}
1 &\text{if } 0<i_1<\dots<i_k\leq r\Longrightarrow \rk(F_{i_1}\cup\dots\cup F_{i_k})>k,\\
0 &\text{else.}
\end{cases}
\]
\end{result}

As a final application of our developments, we provide a new proof of Poincar\'e duality for $A^*(\sM)$, a result that was first proved by Adiprasito, Huh, and Katz \cite[Theorem~6.19]{AHK}.

\begin{result}[Theorem~\ref{thm:poincareduality}]
Let $\sM$ be a matroid of rank $r+1$. Then for any $k\in 0,\dots,r$, we have an isomorphism of $\Z$-modules:
\begin{align*}
A^k(\sM)&\rightarrow A^{r-k}(\sM)^\vee\\
\gamma&\mapsto (\mu\mapsto\deg_\sM(\mu\gamma)).
\end{align*}
\end{result}

To prove Result G, we simply use our computational algorithm to show that the transformation is lower triangular when written in terms of a particular ordering of the Feichtner--Yuzvinsky basis (see \cite[Corollary~1]{FY}) for $A^k(\sM)$ and its dual basis for $A^{r-k}(\sM)^\vee$, with all diagonal entries equal to $\pm 1$.

\subsection{Related work}

As should be clear from the discussion above, this work is closely related and indebted to prior contributions of several groups of mathematicians. The matroid psi classes that we introduce in this work are built from two special psi classes: $\psi_0=\psi_\emptyset^+$ and $\psi_\infty=\psi_E^-$. These two classes have already been studied extensively by Adiprasito, Huh, and Katz \cite{AHK}, where they were denoted $\beta$ and $\alpha$, respectively. Furthermore, as we mentioned above, the psi classes $\psi_F^-$ played an integral role in the work of Backman, Eur, and Simpson \cite{BES}, where they were denoted $h_F$. Our choice to use different notation for these classes in this paper simply stems from our goal of highlighting the parallel between Chow rings of matroids and Chow rings of moduli spaces of curves.

There is a related notion of ``tropical 	psi classes'' developed by Kerber and Markwig \cite{KerberMarkwig}---these classes form the tropical analogue of the classical psi classes on $\M_{0,n}$. Using the description of $\M_{0,n}$ as a wonderful compactification of the complement of the braid arrangement, due to DeConcini and Processi \cite{DP}, tropical psi classes can be interpreted as special elements of Chow rings of braid matroids with minimal building sets. We note that Chow rings of matroids with building sets were defined by Feichtner and Yuzvinsky \cite{FY} and are more general than the matroid Chow rings studied herein, which correspond to the special case of maximal building sets. It would be very interesting to develop a general theory of psi classes associated to matroids with building sets that simultaneously generalizes the matroid psi classes developed in this paper and the tropical psi classes developed by Kerber and Markwig.

\subsection{Outline of the paper}

Losev-Manin moduli spaces are the setting in which Chow rings of matroids intersect Chow rings of moduli spaces of curves. Because of this, we start this paper with an overview of the definition and key properties of psi classes in Losev-Manin spaces; this is the content of Section~\ref{sec:LMspaces}. We conclude Section~\ref{sec:LMspaces} by using psi classes to recover two known formulas for the volumes of generalized permutahedra, due to Postnikov \cite{Postnikov} and Eur \cite{Eur}. The impetus for this work was the observation that, upon generalizing psi classes to matroids, these proofs work nearly verbatim to compute volume polynomials in the more general matroid context.

In Section~\ref{sec:matroidpsi}, we introduce matroid psi classes, prove the natural generalizations of the properties discussed in Section~\ref{sec:LMspaces}, and then we give new proofs of the results of Eur and Backman, Eur and Simpson, generalizing the volume computations from Section~\ref{sec:LMspaces}, and we also give a new proof of Poincar\'e duality. We note that Section~\ref{sec:matroidpsi} is entirely self-contained and the matroid enthusiast may choose to skip Section~\ref{sec:LMspaces}. On the other hand, we hope that the discussion of Losev-Manin spaces will help the reader understand the context and motivation for the definition and development of matroid psi classes, and that this discussion might even motivate the interested combinatorialist to learn a little more about the beautiful subject of Chow rings of moduli spaces of curves.

\subsection{Acknowledgements}

This paper was born out of the first author's Master's thesis, which was advised by both Federico Ardila and the second author. The authors would like to warmly acknowledge Federico's guidance and contributions to this project. Despite his important influence on this work, Federico generously and graciously decided not to be listed as a coauthor of this paper.

The authors would also like to express their gratitude to the Department of Mathematics at San Francisco State University, and especially to Serkan Ho\c{s}ten and Eric Hsu, for their support and leadership during the tumultuous times of the COVID-19 pandemic, during which much of this work was carried out.

The second author was supported by a San Francisco State University Presidential Award during Fall 2020, and this work was also supported by a grant from the National Science Foundation (RUI DMS-2001439).

\section{Losev-Manin spaces and psi classes}\label{sec:LMspaces}

In order to motivate matroid psi classes, we begin with a discussion of psi classes in the setting of Losev-Manin spaces. Our purpose in this section is to describe the key properties of psi classes that are useful in computations in order to motivate the properties that we require upon generalizing psi classes to matroid Chow rings. The results in this section are well-known, so we do not provide complete proofs, only remarking on where the proofs can be found (or derived) in the literature. At the end of this section, we show how psi classes can be used to compute formulas for volumes of generalized permutahedra. All of the definitions and results in this section will be combinatorially generalized to matroid Chow rings in the next section.

\subsection{Losev-Manin spaces}

Losev-Manin spaces, introduced in \cite{LM}, parametrize collections of points on chains of projective lines. To describe these spaces, let us first establish some terminology. 

A \emph{chain of projective lines of length $k$} is a complex variety of the form
\[
C=C_1\sqcup\dots\sqcup C_k/\sim
\]
where $C_i=\P^1$ for all $i=1,\dots,k$ and $\sim$ is the relation that identifies $\infty_i=[0:1]\in C_i$ with $0_{i+1}=[1:0]\in C_{i+1}$ to form a node. The projective lines $C_1,\dots,C_k$ are referred to as the \emph{components} of the chain $C$, and we define $0=0_1\in C_1$ and $\infty=\infty_k\in C_k$.

Given a chain of projective lines $C$, a configuration of $n$ points $p_1,\dots,p_n\in C$ is called \emph{stable} if $\{p_1,\dots,p_n\}$ is disjoint from $0$, $\infty$, and the nodes of $C$, and if each component of $C$ contains at least one $p_i$. We do not require the points to be distinct. Two stable configurations $(C;p_1,\dots,p_n)$ and $(C';p_1',\dots,p_n')$ are said to be \emph{isomorphic} if there exists an isomorphism of varieties $f:C\rightarrow C'$ such that $f(0)=0$, $f(\infty)=\infty$, and $f(p_i)=p_i'$ for all $i=1,\dots,n$. 

\begin{definition}
For any $n\geq 1$, the \emph{Losev-Manin space $\LM_n$} is the set of all stable configurations of $n$ points on chains of projective lines, up to isomorphism. A point in $\LM_n$ is an equivalence class $[C,p_1,\dots,p_n]$ where $C$ is a chain of projective lines and $p_1,\dots,p_n\in C$ is a stable configuration of $n$ points.
\end{definition}

The sets $\LM_n$ were first constructed as smooth projective varieties by Losev and Manin \cite{LM}; in fact, they proved that $\LM_n$ is the toric variety associated to the $(n-1)$-dimensional permutahedron. In particular, $\LM_n$ is a disjoint union of tori, one corresponding to each face of the permutahedron. We now describe those tori explicitly.

To every flag of nonempty subsets
\[
\F=(\emptyset=F_0\subsetneq F_1\subsetneq \dots\subsetneq F_k\subsetneq F_{k+1}=[n])
\]
define a subset of $\LM_n$ by
\[
\T_\F=\bigg\{[C;p_1,\dots,p_n]\;\bigg|\;{C \text{ has } k+1 \text{ components } C_0,\dots,C_{k} \atop \text{and } p_j\in C_i \text{ if and only if }j\in F_{i+1}\setminus F_{i}}\bigg\}.
\]
We depict a general element of $\T_\F$ as follows:
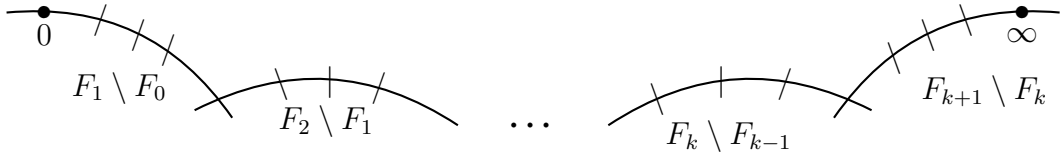
\begin{center}
\tikz{
\draw[thick] (-7,0.1) edge[-,bend right] (-10,1.5);
\draw[thick] (-7.5,0.2) edge[-,bend left] (-4,0);
\draw[thick] (-2,0) edge[-,bend left] (1.5,0.2);
\draw[thick] (1,0.1) edge[-,bend left] (4,1.5);
\node [] at (-3,0) {$\boldsymbol{\cdots}$};
\node [] at (-9.5,1.5) {$\bullet$};
\node [] at (-9.5,1.2) {$0$};
\node [] at (-8.75,1.4) {$/$};
\node [] at (-8.25,1.2) {$/$};
\node [] at (-7.85,1) {$/$};
\node [] at (-8.5,0.5) {$F_1\setminus F_0$};
\node [] at (-6.35,0.55) {$\backslash$};
\node [] at (-5.70,0.58) {$|$};
\node [] at (-5.05,0.5) {$/$};
\node [] at (-5.75,0.05) {$F_2\setminus F_1$};
\node [] at (0.35,0.55) {$/$};
\node [] at (-0.5,0.58) {$|$};
\node [] at (-1.35,0.35) {$\backslash$};
\node [] at (-0.40,-0.15) {$F_{k}\setminus F_{k-1}$};
\node [] at (3.5,1.5) {$\bullet$};
\node [] at (3.5,1.2) {$\infty$};
\node [] at (2.75,1.4) {$\backslash$};
\node [] at (2.25,1.2) {$\backslash$};
\node [] at (1.80,0.9) {$\backslash$};
\node [] at (3,0.45) {$F_{k+1}\setminus F_{k}$};
}
\end{center}

Notice that every element of $\LM_n$ is an element of exactly one set of the form $\T_\F$, so the sets $\T_\F$ partition $\LM_n$. Moreover each $\T_\F$ is an algebraic torus. To see why, consider a particular $\T_\F$ and choose one point from each set $F_{i+1}\setminus F_{i}$. Notice that there is a unique automorphism of $C$ that maps the chosen point in $F_{i+1}\setminus F_{i}$ to $[1,1]\in C_i$. After fixing this isomorphism, the remaining points in $F_{i+1}\setminus F_{i}$ can vary throughout any point of $C_i$ except $0_i$ and $\infty_i$. It follows that
\[
\T_\F=(\C^*)^{|F_1|-|F_0|-1}\times(\C^*)^{|F_2|-|F_1|-1}\dots\times(\C^*)^{|F_{k+1}|-|F_k|-1}=(\C^*)^{n-k-1}.
\]
The tori $\T_\F$ are not closed subvarieties of $\LM_n$, but we may take their closures, which leads to the following important subvarieties.

\begin{definition}
The \emph{stratum $X_\F\subseteq\LM_n$} associated to a flag $\F$ of subsets of $[n]$ is the Zariski closure of the torus $\T_\F$:
\[
X_\F=\overline{\T_\F}.
\]
We say that a subvariety $Z\subseteq\LM_n$ is a \emph{stratum} if it is equal to $X_\F$ for some flag $\F$. For a subset $\emptyset\subsetneq F\subsetneq [n]$, we use the shorthand
\[
X_F=X_{\emptyset\subsetneq F\subsetneq[n]}.
\]
\end{definition}

Each stratum is, again, a disjoint union of tori. To describe these inclusions, it is useful to introduce the notion of refinements. We say that a flag
\[
\F'=(\emptyset\subsetneq F_1'\subsetneq \dots\subsetneq F_\ell'\subsetneq [n])
\]
is a refinement of the flag
\[
\F=(\emptyset\subsetneq F_1\subsetneq \dots\subsetneq F_k\subsetneq [n])
\]
and write $\F'\preceq\F$ if, for every $i\in\{1,\dots,k\}$, there exists some $j\in\{1,\dots,\ell\}$ such that $F_i=F_j'$. With this notion, it can be checked that
\[
X_\F=\displaystyle\bigsqcup_{\F'\preceq\F} \T_{\F'}.
\]
In particular, it follows that $X_{\F_1}\cap X_{\F_2}=X_{\F_3}$ where $\F_3$ is the maximal common refinement of $\F_1$ and $\F_2$ (the intersection is empty if no common refinement exists). 

\subsection{Chow rings and volumes of generalized permutahedra}

The Chow ring of $\LM_n$ is well-known and can be expressed as a quotient of the formal polynomial ring generated by $X_F$ with $F$ a proper subset of $[n]$. By general results in toric geometry \cite[Theorem~12.5.3]{Toric}, we have
\begin{equation}\label{eq:chowring}
A^*(\LM_n)=\frac{\Z\big[X_F\;|\;\emptyset \subsetneq F\subsetneq [n]\big]}{\I+\J}
\end{equation}
where the ideals $\I$ and $\J$ are defined by
\[
\I=\big\langle X_FX_G\;|\;F\text{ and }G\text{ are incomparable} \big\rangle
\]
and
\[
\J=\bigg\langle\sum_{i\in F}X_F-\sum_{j\in F}X_F\;\Big|\;i,j\in[n] \bigg\rangle.
\]
The generators $D_F=[X_F]\in A^1(\LM_n)$, are called \emph{boundary divisors}. The Chow ring has a natural grading by codimension
\[
A^*(\LM_n)=\displaystyle\bigoplus_{k=0}^{n-1}A^k(\LM_n)
\] 
and a degree map
\[
\deg_{\LM_n}:A^{n-1}(\LM_n)\rightarrow \Z,
\]
which is a linear isomorphism uniquely determined by the property that the degree of the class of any point is one.

Any divisor $D\in A^1(\LM_n)$ can be written in the form
\[
D=D(x)=\sum_{\emptyset\subsetneq F\subsetneq [n]}x_FD_F\in A^1(\LM_n)
\]
with $x_F\in\Z$ and, in this setting, $D(x)$ is nef if and only if the numbers $x_F$ are submodular, meaning that, for all $F_1,F_2\subseteq[n]$, we have
\begin{equation}\label{eq:submodular}
x_{F_1}+x_{F_2}\geq x_{F_1\cap F_2}+x_{F_1\cup F_2},
\end{equation}
where, by convention, we always assume $x_\emptyset=x_{[n]}=0$. Given a nef divisor $D(x)$, we consider the corresponding polytope $\Pi_n(x)\subseteq\R^n$ defined by
\begin{equation}\label{eq:permhyperplanes}
t_1+\dots+t_n=0\;\;\;\text{ and }\;\;\;\sum_{i\in F}{t_i}\leq x_F\;\;\;\text{ for all }\;\;\;\emptyset\subsetneq F\subsetneq [n].
\end{equation}
These polytopes were studied under the name of \emph{generalized permutahedra} by Postnikov \cite{Postnikov}, wherein several formulas for their volumes were discovered and proved (see Theorem~\ref{thm:postnikov} below). 

By standard results in toric geometry (\cite[Theorem 13.4.3]{Toric}), the volumes of generalized permutahedra can also be derived by computations in the Chow ring:
\begin{equation}\label{eq:volumepolynomialdivisors}
\Vol(\Pi_n(x))=\frac{1}{(n-1)!}\deg_{\LM_n}(D(x)^{n-1}).
\end{equation}
In order to utilize \eqref{eq:volumepolynomialdivisors}, one needs to expand the product $D(x)^{n-1}$, then use relations in $\I$ and $\J$ to write the result as a linear combination of products of the form $D_{F_1}\dots D_{F_{n-1}}$ where the indexing sets form a complete flag
\[
\emptyset\subsetneq F_1\subsetneq \dots\subsetneq F_{n-1}\subsetneq [n], 
\]
then use the fact that, for any complete flag, $\deg_{\LM_n}(D_{F_1}\dots D_{F_{n-1}})=1$. This process was carried out in the more general matroid context by Eur \cite{Eur}, which led to a new formula for volumes of generalized permutahedra (see Theorem~\ref{thm:eur} below). The heart of Eur's argument is figuring out how to systematically express general products of divisors in terms of products of divisors indexed by complete flags. Phrased another way, the difficulty in this computation is dealing with self-intersections of divisors. In the context of Losev-Manin spaces, there is a useful tool for just this type of self-intersection: psi classes.

\subsection{Psi classes on Losev-Manin spaces}

To understand the utility of psi classes, it is useful to discuss the multiplicative structure of $A^*(\LM_n)$. If $F$ and $G$ are two distinct proper subsets of $[n]$, then the corresponding subvarieties $X_F$ and $X_G$ either intersect transversally, or they don't intersect at all. In particular, if $F$ and $G$ are distinct, then
\[
D_FD_G=\begin{cases}
[X_{\emptyset\subsetneq F\subsetneq G\subsetneq [n]}] & \text{if } F\subsetneq G,\\
[X_{\emptyset\subsetneq G\subsetneq F\subsetneq [n]}] & \text{if } G\subsetneq F,\\
0 & \text{ if }F\text{ and }G\text{ are incomparable}.
\end{cases}
\]
More generally, if $F_1,\dots,F_k\subseteq [n]$ are distinct subsets, we have
\[
D_{F_1}\cdots D_{F_k}=\begin{cases}
[X_\F] & \parbox[]{8cm}{if, after possibly relabeling, $F_1,\dots,F_k$ form a flag $\F=(\emptyset\subsetneq F_1\subsetneq\dots\subsetneq F_k\subsetneq [n])$,}\\
0 & \text{if }F_i\text{ and }F_j\text{ are incomparable for some }i,j.
\end{cases}
\]
For convenience, for any flag $\F=(\emptyset\subsetneq F_1\subsetneq\dots\subsetneq F_k\subsetneq [n])$, we define
\[
D_\F=[X_\F]=D_{F_1}\dots D_{F_k}\in A^k(\LM_n).
\]
The main question, then, is: How do we multiply divisors when they are not all indexed by distinct subsets? This is where psi classes are useful. In the setting of Losev-Manin spaces, there are two basic psi classes upon which the others are built.

\begin{definition}
Let $n\geq 1$. The \emph{psi class} $\psi_0\in A^1(\LM_n)$ is the first Chern class of the line bundle $\bL_0$, whose fiber over a point $[C,p_1,\dots,p_n]\in\LM_n$ is the cotangent line of $C$ at $0$. The psi class $\psi_\infty\in A^1(\LM_n)$ is the first Chern class of the line bundle $\bL_\infty$, whose fiber over a point $[C,p_1,\dots,p_n]\in\LM_n$ is the cotangent line of $C$ at $\infty$. 
\end{definition}

A  more combinatorial characterization of psi classes, which will be our starting point for the matroid generalization, appears in Lemma~\ref{lem:psilinearcombo} below. To understand why the psi classes are useful for computing self-intersections, we require a bit of additional notation. For a finite set $F$, let $\LM_F$ denote the Losev-Manin space with marked points indexed by $F$. Of course, $\LM_{[n]}=\LM_n$. If $|F|>2$, then for each $i\in F$, there is a \emph{forgetful map}
\[
f_i:\LM_F\rightarrow\LM_{F\setminus\{i\}}.
\]
For each point $[C;(p_j)_{j\in F}]\in\LM_F$, the function $f_i$ forgets the marked point $p_i$ and then, if the component that contained $p_i$ no longer has any marked points, it contracts that entire component to a single point. The second step is necessary in order to insure that the image of $f$ is a stable configuration. 

More generally, if $\emptyset\subsetneq G\subseteq F$, then there is a forgetful map
\[
r_G:\LM_F\rightarrow\LM_G.
\]
To define this map, label the points $F\setminus G=\{i_1,\dots,i_k\}$ and define
\[
r_G=f_{i_1}\circ\dots\circ f_{i_k}.
\]
In other words, $r_G$ forgets the points that are not in $G$. We use the letter $r$ for ``remember'' because the map $r_G$ remembers the points in the index set $G$. It follows from the definition that the order of the composition in the definition of $r_G$ is irrelevant, and if $\emptyset\subsetneq G_1\subseteq G_2\subseteq F$, then
\begin{equation}\label{eq:composerememberpsi}
r_{G_1}=r_{G_1}\circ r_{G_2}.
\end{equation}
Using the forgetful maps, we obtain a more general set of psi classes.

\begin{definition}\label{def:generalpsi}
For $n\geq 1$ and $\emptyset\subseteq F\subseteq [n]$, define classes $\psi_F^-,\psi_F^+\in A^1(\LM_n)$ by
\[
\psi_F^-=r_F^*(\psi_\infty)\;\;\;\text{ and }\;\;\;\psi_F^+=r_{F^c}^*(\psi_0),
\]
where $r_F^*$ is the pullback of $r_F:\LM_n\rightarrow\LM_F$ and $F^c=[n]\setminus F$.
\end{definition}

Notice that $\psi_0=\psi_{\emptyset}^+$ and $\psi_\infty=\psi_{[n]}^-$. The reason we introduce psi classes is because they naturally arise when self-intersecting divisors in the following way.

\begin{lemma}\label{lem:selfintersection}
If $F$ is a proper subset of $[n]$, then
\[
D_F^2=D_F(-\psi_F^--\psi_F^+)\in A^2(\LM_n).
\]
\end{lemma}

\begin{proof}[Proof sketch]
This follows from the observation (see, for example, \cite[Lemma 25.2.2]{Mirror}) that the normal bundle of $X_F$ in $\LM_n$ is
\[
g_F^*(r_F^*(\bL_\infty^\vee)\otimes r_{F^c}^*(\bL_0^\vee)),
\]
where $g_F:X_F\rightarrow\LM_n$ is the inclusion.
\end{proof}

In particular, Lemma~\ref{lem:selfintersection} allows us to compute any product of boundary divisors in terms of psi classes. We have the following immediate corollary.

\begin{corollary}\label{cor:intersectdivisors}
If $F_1,\dots,F_k\subseteq [n]$ are distinct proper subsets and $d_1,\dots,d_k$ are positive integers, then
\[
D_{F_1}^{d_1}\cdots D_{F_k}^{d_k}=\begin{cases}
D_\F\displaystyle\prod_{i=1}^k(-\psi_{F_i}^--\psi_{F_i}^+)^{d_i-1} & \parbox[]{8cm}{if, after possibly relabeling, $F_1,\dots,F_k$ form a flag $\F=(\emptyset\subsetneq F_1\subsetneq\dots\subsetneq F_k\subsetneq [n])$,}\\
0 & \text{if }F_i\text{ and }F_j\text{ are incomparable for some }i,j.
\end{cases}
\]
\end{corollary}

In order to utilize psi classes in the volume computation of Equation~\ref{eq:volumepolynomialdivisors}, it remains to understand how to compute the degree of expressions of the form in Corollary~\ref{cor:intersectdivisors}. The next result reduces these computations to computing degrees of monomials in $\psi_0$ and $\psi_\infty$.

\begin{lemma}\label{lem:product}
If $\F=(\emptyset = F_0\subsetneq F_1\subsetneq\dots\subsetneq F_k\subsetneq F_{k+1}=[n])$ is a flag of subsets and $a_0^+,a_1^-,a_1^+,\dots,a_{k}^-,a_k^+,a_{k+1}^-$ are nonnegative integers, then
\[
\deg_{\LM_n}\bigg(D_\F\prod_{i=0}^{k}(\psi_{F_{i}}^+)^{a_{i}^+}(\psi_{F_{i+1}}^-)^{a_{i+1}^-}\bigg)=\prod_{i=0}^{k}\deg_{\LM_{F_{i+1}\setminus F_i}}\big(\psi_0^{a_{i}^+}\psi_\infty^{a_{i+1}^-}\big).
\]
\end{lemma}

Pictorially, we think of the psi classes $\psi_{F_i}^\pm$ as being associated to the left and right side of the node indexed by $F_i$:
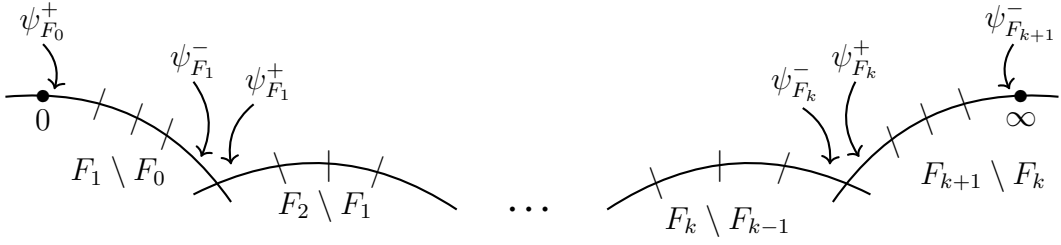
\begin{center}
\tikz{
\draw[thick] (-7,0.1) edge[-,bend right] (-10,1.5);
\draw[thick] (-7.5,0.2) edge[-,bend left] (-4,0);
\draw[thick] (-2,0) edge[-,bend left] (1.5,0.2);
\draw[thick] (1,0.1) edge[-,bend left] (4,1.5);
\node [] at (-3,0) {$\boldsymbol{\cdots}$};
\node [] at (-9.5,1.5) {$\bullet$};
\node [] at (-9.5,1.2) {$0$};
\node [] at (-8.75,1.4) {$/$};
\node [] at (-8.25,1.2) {$/$};
\node [] at (-7.85,1) {$/$};
\node [] at (-8.5,0.5) {$F_1\setminus F_0$};
\node [] at (-6.35,0.55) {$\backslash$};
\node [] at (-5.70,0.58) {$|$};
\node [] at (-5.05,0.5) {$/$};
\node [] at (-5.75,0.05) {$F_2\setminus F_1$};
\node [] at (0.35,0.55) {$/$};
\node [] at (-0.5,0.58) {$|$};
\node [] at (-1.35,0.35) {$\backslash$};
\node [] at (-0.40,-0.15) {$F_{k}\setminus F_{k-1}$};
\node [] at (3.5,1.5) {$\bullet$};
\node [] at (3.5,1.2) {$\infty$};
\node [] at (2.75,1.4) {$\backslash$};
\node [] at (2.25,1.2) {$\backslash$};
\node [] at (1.80,0.9) {$\backslash$};
\node [] at (3,0.45) {$F_{k+1}\setminus F_{k}$};
\node [] at (-9.5,2.5) {$\psi_{F_0}^+$};
\draw[thick] (-9.5,2.2) edge[->,bend left] (-9.35,1.6);
\node [] at (-7.5,2) {$\psi_{F_1}^-$};
\node [] at (-6.5,1.7) {$\psi_{F_1}^+$};
\draw[thick] (-7.5,1.7) edge[->,bend left] (-7.4,.7);
\draw[thick] (-6.7,1.4) edge[->,bend right] (-7,.6);
\node [] at (.5,1.7) {$\psi_{F_k}^-$};
\node [] at (1.3,2) {$\psi_{F_k}^+$};
\draw[thick] (.6,1.35) edge[->,bend left] (.96,.6);
\draw[thick] (1.3,1.7) edge[->,bend right] (1.35,.7);
\node [] at (3.5,2.5) {$\psi_{F_{k+1}}^-$};
\draw[thick] (3.4,2.2) edge[->,bend right] (3.35,1.6);
}
\end{center}
The products in Lemma~\ref{lem:product} are over all of the components of the curves, which should help explain the indices in the products.

\begin{proof}[Proof sketch of Lemma~\ref{lem:product}]
Let $g_\F: X_\F\rightarrow \LM_n$ be the inclusion. By the projection formula,
\begin{equation}\label{eq:degreeproduct}
\deg_{\LM_n}\bigg(D_\F\prod_{i=0}^{k}(\psi_{F_{i}}^+)^{a_{i}^+}(\psi_{F_{i+1}}^-)^{a_{i+1}^-}\bigg)=\deg_{X_\F}\bigg(g_\F^*\bigg(\prod_{i=0}^{k}(\psi_{F_{i}}^+)^{a_{i}^+}(\psi_{F_{i+1}}^-)^{a_{i+1}^-}\bigg)\bigg)
\end{equation}
Notice that
\[
X_\F=\prod_{i=0}^{k}\LM_{F_{i+1}\setminus{F_{i}}}.
\]
If $p_i:X_\F\rightarrow\LM_{F_{i+1}\setminus{F_{i}}}$ is the projection onto the $i$th component of this product, then 
\[
g_\F^*(\psi^+_{F_{i}})=p_{i}^*(\psi_0) \;\;\;\text{ and }\;\;\; g_\F^*(\psi^-_{F_{i+i}})=p_{i}^*(\psi_\infty).
\]
Thus, the degree in the right-hand side of \eqref{eq:degreeproduct} can be computed as a product of degrees on each factor:
\begin{align*}
\deg_{X_\F}\bigg(g_\F^*\bigg(\prod_{i=0}^{k}(\psi_{F_{i}}^+)^{a_{i}^+}(\psi_{F_{i+1}}^-)^{a_{i+1}^-}\bigg)\bigg)
&=\deg_{X_\F}\bigg(\prod_{i=0}^{k}p_i^*\Big(\psi_0^{a_{i-1}^+}\psi^{a_i^-}_\infty\Big)\bigg)\\
&=\prod_{i=0}^{k}\deg_{\LM_{F_{i+1}\setminus F_{i}}}\Big(\psi_0^{a_{i}^+}\psi^{a_{i+1}^-}_\infty\Big).\qedhere
\end{align*}
\end{proof}

Lastly, we simply need to know how to compute degrees of monomials in $\psi_0$ and $\psi_\infty$. The next result accomplishes that.

\begin{lemma}\label{lem:degreeofpsi}
If $n>1$ and $a$ and $b$ are nonnegative integers, then
\[
\deg_{\LM_n}(\psi_0^{a}\psi_\infty^{b})={n-1 \choose a,b},
\]
where, for any nonnegative integers $k,\ell,m$,
\[
{m\choose k,\ell}=
\begin{cases}
{m\choose k}={m\choose \ell}=\frac{m!}{k!\ell!} &\text{if }k+\ell=m,\\
0&\text{else.}
\end{cases}
\]
\end{lemma}

\begin{proof}[Proof sketch]
To our knowledge, this exact result is not stated in the literature anywhere. However, it is well known and follows, using the results of \cite{AG}, from the same arguments used to compute degrees of monomials of psi classes on $\overline{\mathcal{M}}_{0,n}$ (see \cite[Section 25.2]{Mirror}). In the specific setting of $\LM_n$, this result is given as Exercise 52 in \cite{Cavalieri}. 
\end{proof}

The combination of the previous three results tell us everything we need to know about effectively computing degrees of products of boundary divisors, such as those that appear in the right-hand side of \eqref{eq:volumepolynomialdivisors}. We illustrate such a computation in the next example.

\begin{example}
Let $n=7$ and consider the sets
\[
F_1=\{1,2\},\;\;\; F_2=\{1,2,3,4,5\},\;\;\;\text{ and }\;\;\; F_3=\{1,2,3,4,5,6\}.
\]
Let us compute $\deg_{\LM_7}(D_{F_1}^3D_{F_2}^2D_{F_3})$. By Corollary~\ref{cor:intersectdivisors}, we have
\[
D_{F_1}^3D_{F_2}^2D_{F_3}=D_{F_1}D_{F_2}D_{F_3}(-\psi_{F_1}^--\psi_{F_1}^+)^2(-\psi_{F_2}^--\psi_{F_2}^+).
\]
Expanding the polynomial, we obtain
\[
-D_{F_1}D_{F_2}D_{F_3}\Big((\psi_{F_1}^-)^2\psi_{F_2}^-+2\psi_{F_1}^-\psi_{F_1}^+\psi_{F_2}^-+(\psi_{F_1}^+)^2\psi_{F_2}^-+(\psi_{F_1}^-)^2\psi_{F_2}^++2\psi_{F_1}^-\psi_{F_1}^+\psi_{F_2}^++(\psi_{F_1}^+)^2\psi_{F_2}^+\Big).
\]
Using Lemmas~\ref{lem:product} and \ref{lem:degreeofpsi}, we see that the degree of the first monomial is zero, because the first term in the product of binomials is
\[
{2-0-1\choose 0,2}={1\choose 0,2}=0.
\]
By a similar argument, the degree is zero on all of the monomials except for the second one. The degree of the second monomial is
\begin{align*}
\deg\Big(-D_{F_1}D_{F_2}D_{F_3}2\psi_{F_1}^-\psi_{F_1}^+\psi_{F_2}^-\Big)&=-2{2-0-1\choose 0,1}{5-2-1\choose 1,1}{6-5-1\choose 0,0}{7-6-1\choose 0,0}\\
&=-2(1)(2)(1)(1)=-4.
\end{align*}
Thus, we conclude that $\deg_{\LM_7}(D_{F_1}^3D_{F_2}^2D_{F_3})=-4.$
\end{example}

Since our ultimate goal is to generalize psi classes to the combinatorial setting of matroids, we present one final result, which characterizes the psi classes as linear combinations of boundary divisors.

\begin{lemma}\label{lem:psilinearcombo}
For any subset $F\subseteq [n]$ and any $i\in[n]$,
\[
\psi_F^-=\sum_{\emptyset\subsetneq G\subsetneq[n]\atop i\in G} D_G-\sum_{\emptyset\subsetneq G\subsetneq[n]\atop G\supseteq F}D_G\;\;\;\text{ and }\;\;\;\psi_F^+=\sum_{\emptyset\subsetneq G\subsetneq[n]\atop i\notin G} D_G-\sum_{\emptyset\subsetneq G\subsetneq[n]\atop G\subseteq F}D_G.
\]
In particular, taking $F=\emptyset$ and $F=[n]$, respectively, we obtain
\[
\psi_0=\sum_{\emptyset\subsetneq G\subsetneq[n]\atop i\notin G} D_G\;\;\;\text{ and }\;\;\;\psi_\infty=\sum_{\emptyset\subsetneq G\subsetneq[n] \atop i\in G} D_G.
\]
\end{lemma}

\begin{proof}[Proof sketch]
The formulas for $\psi_0$ and $\psi_\infty$ follow from repeated application of the comparison lemma:
\[
 f_i^*(\psi_0)=\psi_0+D_{\{i\}}\;\;\;\text{ and } f_i^*(\psi_\infty)=\psi_\infty+D_{\{i\}^c},
\]
and the fact that $\psi_0=\psi_\infty=0\in A^*(\LM_{\{i\}})$. See \cite[Theorem 5.8]{AG} or \cite[Lemma 10]{Cavalieri} for a discussion of the comparison lemma in the setting of Losev-Manin spaces. The formulas for $\psi_F^{\pm}$ then follow from their definition in terms of forgetful maps along with the observation that
\[
r_F^*(D_G)=\sum_{\emptyset\subsetneq G'\subsetneq [n]\atop G\subseteq G'\subseteq G\cup F^c}D_{G'}.\qedhere
\]
\end{proof}

We now illustrate the utility of psi classes by showing how the results reviewed above lead to new proofs of two previously-known formulas for volumes of generalized permutahedra.

\subsection{Psi classes and Eur's volume formula}

Eur recently proved the following formula for volumes of generalized permutahedra.

\begin{theorem}[\cite{Eur} Proposition 4.2]\label{thm:eur}
If $\{x_F\in\Z\;|\;\emptyset\subsetneq F\subsetneq [n]\}$ is submodular, then
\[
\Vol(\Pi_n(x))=\frac{1}{(n-1)!}\sum_{F_1,\dots,F_{k}\atop d_1,\dots,d_k}(-1)^{n-k-1}{n-1\choose d_1,\dots,d_k}\prod_{i=1}^k{d_i-1\choose\tilde d_i-|F_i|}{|F_{i+1}|-|F_i|-1\choose \tilde d_i-|F_i|}x_{F_i}^{d_i}
\]
where the sum is over flags of subsets $\emptyset\subsetneq F_1\subsetneq\dots\subsetneq F_k\subsetneq F_{k+1}=[n]$ and positive integers $d_1,\dots,d_k$ such that $d_1+\dots+d_k=n-1$, and the numbers $\tilde d_j$ are defined by
\[
\tilde d_j=\sum_{i=1}^j d_i.
\] 
\end{theorem}

In fact, Eur generalized and proved this formula in a more general matroid setting, which we will discuss in the next section. For now, let us give a short proof of Theorem~\ref{thm:eur} using psi classes.

\begin{proof} Applying \eqref{eq:volumepolynomialdivisors}, we have
\begin{align*}
\Vol(\Pi_n(x))&=\frac{1}{(n-1)!}\deg\left(\bigg(\displaystyle\sum_{\emptyset\subsetneq F\subsetneq [n]}x_FD_F\bigg)^{n-1}\right)\\
&=\frac{1}{(n-1)!}\sum_{F_1,\dots,F_k \atop d_1,\dots,d_k} {n-1 \choose d_1,\dots,d_k}\deg(D_{F_1}^{d_1}\dots D_{F_k}^{d_k})x_{F_1}^{d_1}\dots x_{F_k}^{d_k},
\end{align*}
where the sum is over $k$-tuples of distinct proper subsets $\emptyset\subsetneq F_1,\dots,F_k\subsetneq [n]$ and positive integers $d_1,\dots,d_k$ that sum to $n-1$. Since $D_{F_1}\cdots D_{F_k}=0$ when the indexing sets cannot be rearranged into a flag, we can restrict the sum to be over all flags of subsets of the form $\F=(\emptyset\subsetneq F_1\subsetneq\dots\subsetneq F_k\subsetneq F_{k+1}=[n])$. For such a flag, we may apply Corollary~\ref{cor:intersectdivisors} to obtain
\begin{align*}
\deg(D_{F_1}^{d_1}\dots D_{F_k}^{d_k})&=\deg\Big(D_\F\prod_{i=1}^k(-\psi_{F_i}^--\psi_{F_i}^+)^{d_i-1}\Big)\\
&=\deg\Bigg(D_\F(-1)^{n-k-1}\sum_{a_i^-,a_i^+}\prod_{i=1}^k{d_i-1 \choose a_i^-,a_i^+}(\psi_{F_i}^-)^{a_i^-}(\psi_{F_i}^+)^{a_i^+} \Bigg)\\
&=(-1)^{n-k-1}\sum_{a_i^-,a_i^+}\prod_{i=1}^k{d_i-1 \choose a_i^-,a_i^+}\deg\Big(D_\F\prod_{i=1}^k(\psi_{F_i}^-)^{a_i^-}(\psi_{F_i}^+)^{a_i^+} \Big).
\end{align*}
If we now use Lemmas~\ref{lem:product} and \ref{lem:degreeofpsi} to compute the degree, we obtain
\[
\deg(D_{F_1}^{d_1}\dots D_{F_k}^{d_k})=(-1)^{n-k-1}\sum_{a_i^-,a_i^+}\prod_{i=1}^k{d_i-1 \choose a_i^-,a_i^+}\prod_{i=0}^{k}{|F_{i+1}|-|F_{i}|-1. \choose a_{i}^+,a_{i+1}^-},
\]
where $a_0^+=a_{k+1}^-=0$. In order for the two sets of binomials to be nonzero, there are two systems of equations that $a_i^-$ and $a_i^+$ must satisfy:
\[
a_i^-+a_i^+=d_i-1\;\;\;\text{ for all }i=1,\dots,k
\]
and
\[
a_{i}^++a_{i+1}^-=|F_{i+1}|-|F_{i}|-1\;\;\;\text{ for all }i=0,\dots,k.
\]
Along with the conditions $a_0^+=a_{k+1}^-=0$, there is a unique solution given by
\[
a_i^+=\tilde d_i-|F_{i}|\;\;\;\text{ for all }i=1,\dots,k.
\]
It follows that
\[
\deg(D_{F_1}^{d_1}\dots D_{F_k}^{d_k})=(-1)^{n-k-1}\prod_{i=1}^k{d_i-1 \choose \tilde d_i -|F_i|}\prod_{i=0}^{k}{|F_{i+1}|-|F_{i}|-1. \choose \tilde d_{i}-|F_{i}|}
\]
Eur's formula then follows by noticing that the $i=0$ term in the second product is one. 
\end{proof}

\subsection{Psi classes and Postnikov's volume formula}

A different formula for the volumes of generalized permutahedra had previous been proved by Postnikov \cite{Postnikov}. In order to set up Postnikov's formula, we require a little more notation.

For any nonempty subset $F\subseteq [n]$, define a corresponding simplex
\[
\Delta_F=\Conv\{e_i:i\in F\}\subseteq\R^n.
\]
If $y_F$ is a nonnegative real number for every nonempty subset $F\subseteq [n]$, then Postnikov observed (\cite[Proposition 6.2]{Postnikov}) that the polytope
\[
\Pi^\Delta_n(y)=\sum_{\emptyset\subsetneq F\subseteq [n]}y_F\Delta_F,
\]
where the sum denotes Minkowski summation, consists of all points $(t_1,\dots,t_n)\in\R^n$ such that:
\[
t_1+\dots+t_n=z_{[n]}\;\;\;\text{ and }\;\;\;\sum_{i\in F}t_i\geq z_F\;\;\;\text{ for all }\;\;\;\emptyset\subsetneq F\subsetneq [n].
\]
where $z_F$ and $y_F$ are related by the invertible linear transformation
\[
z_F=\sum_{G\subseteq F} y_G.
\]
Under the transformation $(t_1,\dots,t_n)\mapsto(z_{[n]}-t_1,-t_2,\dots,-t_n)$, notice that $\Pi_n^\Delta(y)$ is identified with $\Pi_n(x)$ (introduced in Equation \eqref{eq:permhyperplanes}), where for any proper subset $\emptyset\subsetneq F\subsetneq [n]$, the variables $x_F$ and $y_F$ are related by
\[
x_F=\begin{cases}
-z_F=-\displaystyle\sum_{G\subseteq F} y_G & \text{if }1\notin F,\\
z_{[n]}-z_F=\displaystyle\sum_{G\subseteq[n]}y_G-\displaystyle\sum_{G\subseteq F} y_G&\text{if }1\in F.
\end{cases}
\]
In addition, it can be checked that, when $y_F\geq 0$ for all nonempty subsets $F$, the corresponding numbers $x_F$ are submodular, in the sense of \eqref{eq:submodular}, meaning that the intersection-theoretic formula \eqref{eq:volumepolynomialdivisors} is valid.

Postnikov proved the following formula for the volume of $\Pi_n^\Delta(y)$, which, by polynomiality of volumes, determines the volume for all generalized permutahedra (this last statement is carefully worked out by Ardila, Benedetti, and Doker \cite{ABD}).

\begin{theorem}[\cite{Postnikov} Corollary 9.4]\label{thm:postnikov}
If $y_G\geq 0$ for all nonempty subsets $G\subseteq [n]$, then
\[
\Vol(\Pi_n^\Delta(y))=\frac{1}{(n-1)!}\sum_{G_1,\dots,G_{n-1}}y_{G_1}\cdots y_{G_{n-1}},
\]
where the sum is over collections of nonempty subsets $G_1,\dots,G_{n-1}\subseteq [n]$ such that, for any $0<i_1<\dots<i_k<n$, we have
\[
|G_{i_1}\cup\dots\cup G_{i_k}|>k.
\]
\end{theorem}

\begin{proof}
Let us prove this formula using psi classes. By Equation \eqref{eq:volumepolynomialdivisors}, we have
\[
\Vol(\Pi_n^\Delta(y))=\frac{1}{(n-1)!}\deg\left(\bigg(\displaystyle\sum_{\emptyset\subsetneq F\subsetneq [n]}x_FD_F\bigg)^{n-1}\right).
\]
Applying the change of variables above, notice that
\begin{align*}
\displaystyle\sum_{\emptyset\subsetneq F\subsetneq [n]}x_FD_F&=-\displaystyle\sum_{\emptyset\subsetneq F\subsetneq [n] \atop 1\notin F}\Big(\displaystyle\sum_{G\subseteq F} y_G \Big)D_F+\displaystyle\sum_{\emptyset\subsetneq F\subsetneq [n]\atop 1\in F}\Big(\displaystyle\sum_{G\subseteq[n]}y_G-\displaystyle\sum_{G\subseteq F} y_G\Big)D_F\\
&=\sum_{\emptyset\subsetneq G\subseteq [n]}y_G\Big(-\sum_{F\supseteq G\atop 1\notin F}D_F+\sum_{F\subseteq [n]\atop 1\in F}D_F-\sum_{F\supseteq G\atop 1\in F}D_F\Big)\\
&=\sum_{\emptyset\subsetneq G\subseteq [n]}y_G\Big(\sum_{F\subseteq [n]\atop 1\in F}D_F-\sum_{F\supseteq G}D_F\Big)\\
&=\sum_{\emptyset\subsetneq G\subseteq [n]}y_G\psi_G^-,
\end{align*}
where the last equality follows from Lemma~\ref{lem:psilinearcombo}. Thus, Postnikov's formula can be reinterpreted as an intersection-theoretic property of psi classes. In particular, Postnikov's formula is equivalent to the statement that
\begin{equation}\label{eq:postnikovpsi}
\deg(\psi_{G_1}^-\dots\psi_{G_{n-1}}^-)=
\begin{cases}
1 &\text{if } 0<i_1<\dots<i_k<n\Longrightarrow |G_{i_1}\cup\dots\cup G_{i_k}|>k,\\
0 &\text{else.}
\end{cases}
\end{equation}

To prove \eqref{eq:postnikovpsi}, we start by proving the second case. Suppose that there exists some $0<i_1<\dots<i_k<n$ such that
\[
G=G_{i_1}\cup\dots\cup G_{i_k}
\]
has at most $k$ elements. By virtue of Equation~\eqref{eq:composerememberpsi}, notice that
\begin{align*}
\psi_{G_{i_1}}^-\dots \psi_{G_{i_k}}^-&=r_{G_{i_1}}^*(\psi_\infty)\cdots r_{G_{i_k}}^*(\psi_\infty)\\
&=r_G^*(r_{G_{i_1}}^*(\psi_\infty)\cdots r_{G_{i_k}}^*(\psi_\infty)).
\end{align*}
Notice that the argument of $r_G^*$ in the final expression is an element of $A^k(\LM_G)$, which is zero because $\dim(\LM_G)=|G|-1$, which we have assume to be strictly less than $k$.

Next, to prove the first case of \eqref{eq:postnikovpsi}, suppose that $0<i_1<\dots<i_k<n$ implies $|G_{i_1}\cup\dots\cup G_{i_k}|>k$. Applying Lemma~\ref{lem:psilinearcombo}, notice that
\begin{align*}
\psi_{G_1}^-\cdots\psi_{G_{n-1}}^-&=\Big(\psi_\infty-\sum_{F\supseteq G_{i_1}}D_F\Big)\cdots\Big(\psi_\infty-\sum_{F\supseteq G_{i_{n-1}}}D_F\Big)\\
&=\sum_{k=0}^{n-1}\psi_\infty^{n-1-k}(-1)^k\sum_{0<i_1<\dots<i_k<n \atop F_j\supseteq G_{i_j}}D_{F_1}\cdots D_{F_k}.
\end{align*}
We claim that the only nonzero term in the sum is the one indexed by $k=0$. To verify this, notice that multiplying $D_{F_1}\cdots D_{F_k}$ will either be zero or a multiple of $D_\F$ for some flag $\F$. In the latter case, notice that the largest set in the flag $\F$ must be $F=F_1\cup\cdots\cup F_k$, which contains $G_{i_1}\cup\dots\cup G_{i_k}$. This implies that $F$ has more than $k$ elements, showing that $n-|F|-1<n-k-1$. It then follows that
\[
\deg(\psi_\infty^{n-k-1}D_{F_1}\dots D_{F_k})=0
\]
because, using Lemmas~\ref{lem:product}, it contains a factor of 
\[
\deg_{\LM_{[n]\setminus F}}\big(\psi_0^{a}\psi_\infty^{n-k-1}\big)=0.
\]
Thus, the only nonzero term in the sum is the one indexed by $k=0$, in which case we compute
\[
\deg(\psi_\infty^{n-1})=1.\qedhere
\]
\end{proof}

\section{Matroid psi classes}\label{sec:matroidpsi}

We now describe a generalization of psi classes from Losev-Manin spaces to the matroid setting. We then use matroid psi classes to give new proofs of formulas for volume polynomials of matroids, and we use them to give a constructive proof of Poincar\'e duality.

\subsection{Matroid basics}\label{sec:matroidbackground}

Before discussing matroid psi classes, we begin by introducing the relevant matroid background and terminology.

\subsubsection{Definitions}

A \emph{matroid} $\sM=(E,\cL)$ consists of a finite set $E$, called the \emph{ground set}, and a collection of subsets $\cL=\cL_\sM\subseteq 2^E$, called \emph{flats}, which satisfy the following two conditions:
\begin{enumerate}
\item if $F_1,F_2$ are flats, then $F_1\cap F_2$ is a flat, and
\item if $F$ is a flat, then every element of $E\setminus F$ is contained in exactly one flat that is minimal among the flats that strictly contain $F$.
\end{enumerate}

Given a matroid $\sM=(E,\cL)$, the set $\cL$ is partially ordered by set inclusion. Furthermore, given any subset $S\subseteq E$, it follows from Property (1) that there is a minimal flat containing $S$, called the \emph{closure} of $S$ and denoted $\cl(S)\in\cL$. Defining the \emph{join} ($\vee$) of two flats to be the closure of their union and the \emph{meet} ($\wedge$) of two flats to be their intersection, it follows from the definitions that the flats $\cL$ form a lattice, called the \emph{lattice of flats} of $\sM$. 

A subset $I\subseteq E$ is called \emph{independent} if, for any $I_1\subsetneq I_2\subseteq I$, we have $\cl(I_1)\subsetneq\cl(I_2)$. The \emph{rank} of a subset $S\subseteq E$, denoted $\rk(S)$, is the size of its largest independent subset. The rank of $\sM$ is defined as the rank of $E$. An alternative characterization of the rank of flats is given by lengths of flags. In particular, the number of nonempty flats in a flag
\[
\F=(\emptyset\subsetneq F_1\subsetneq F_2\subsetneq\cdots\subsetneq F_\ell)
\]
is called the \emph{length} of the flag, denote $\ell(\F)$, and it can be checked from the above definitions that every maximal flag of flats contained in a flat $F$ has length equal to $\rk(F)$.

There are several important types of elements in a matroid $\sM=(E,\cL)$. A \emph{loop} of $\sM$ is an element $e\in E$ such that $\rk(\{e\})=0$, and a \emph{coloop} of $\sM$ is an element $e\in E$ such that $\{e\}^c\in\cL$. Two elements $e,f\in E$ are said to be \emph{parallel} if $\rk(\{e\})=\rk(\{f\})=\rk(\{e,f\})$. A matroid without loops is called \emph{loopless} and a matroid without loops or parallel elements is called \emph{simple}. In other words, a loopless matroid is one for which the empty set is a flat, and a simple matroid is one for which, in addition, each rank-one flat is a singleton.

\subsubsection{Matroid constructions}

Given a matroid $\sM=(E,\cL)$ and a subset $S\subseteq E$, there are several important ways to construct related matroids. The \emph{restriction} of $\sM$ to $S$, denoted $\sM|_S$, is the matroid on ground set $S$ with flats
\[
\cL_{\sM|_S}=\{F\cap S\;|\; F\in\cL_\sM\}.
\]
The \emph{contraction} of $\sM$ by $S$, denoted $\sM/S$ is the matroid on ground set $E\setminus S$ with flats
\[
\cL_{\sM/S}=\{F\setminus S\;|\;F\in\cL_\sM\text{ and }S\subseteq F\}.
\]
Lastly, the \emph{deletion} of $\sM$ by $S$, denoted $\sM\setminus S$, is the restriction of $\sM$ to $E\setminus S$:
\[
\sM\setminus S=\sM|_{E\setminus S}.
\]

If $F,G\in\cL_\sM$, then we introduce the notation $\sM[F,G]=(\sM|_G)/F$. By definition, $\sM[F,G]$ is the matroid of rank $\rk(G)-\rk(F)$ on the ground set $G\setminus F$ with flats
\[
\cL_{\sM[F,G]}=\{H\setminus F\;|\;H\in\cL_\sM,\;F\subseteq H\subseteq G\}.
\] 
Notice that the flats of $\sM[F,G]$ are in natural inclusion-preserving bijection with the flats of $\sM$ that are weakly contained between $F$ and $G$, which comprise the closed interval $[F,G]$. We use the shorthand $\cL[F,G]$ for the flats of $\sM[F,G]$ and we denote the proper flats by $\cL(F,G)$.

Given a matroid $\sM=(E,\cL)$, the \emph{simplificiation} of $\sM$, denoted $\underline\sM$, is the matroid obtained by choosing a distinguished element from each rank-one flat and deleting all other elements of $E$. The simplification is unique up to relabeling the elements of the ground set, so we do not stress the choice of distinguished elements. Notice that the lattice of flats of $\sM$ and $\underline\sM$ are naturally isomorphic.

\subsubsection{Characteristic polynomials} 

Given a matroid $\sM=(E,\cL)$, the \emph{characteristic polynomial} of $\sM$ is defined by
\[
\chi_\sM(\lambda)=\sum_{S\subseteq E}(-1)^{|S|}\lambda^{\rk(E)-\rk(S)}.
\]
From this definition, it is an excellent exercise to check the following three properties.
\begin{enumerate}
\item[($\chi$1)] If $\sM$ has a loop, then $\chi_\sM(\lambda)=0$.
\item[($\chi$2)] If $e$ is a coloop of $\sM$, then $\chi_\sM(\lambda)=(\lambda-1)\chi_{\sM\setminus\{e\}}(\lambda)$.
\item[($\chi$3)] If $e$ is neither a loop nor a coloop, then
\[
\chi_\sM(\lambda)=\chi_{\sM\setminus\{e\}}(\lambda)-\chi_{\sM/\{e\}}(\lambda)
\]
\end{enumerate}
Property ($\chi$3) is called the \emph{deletion-contraction} property, and it generalizes the property of the same name for chromatic polynomials of graphs. Notice that Properties ($\chi$1)--($\chi$3) determine $\chi_\sM(\lambda)$ recursively on the size of the ground set. In addition, it follows from ($\chi$1) and ($\chi$3) that $\chi_\sM(\lambda)=\chi_{\underline\sM}(\lambda)$ for any loopless matroid $\sM$.

It also follows from Properties ($\chi$1)--($\chi$3) that, for any nonempty matroid $\sM$, the characteristic polynomial $\chi_\sM(\lambda)$ is divisible by $\lambda-1$. The \emph{reduced characteristic polynomial} of a nonempty matroid $\sM$ is define by
\[
\ochi_\sM(\lambda)=\frac{\chi_\sM(\lambda)}{\lambda-1}.
\]
Naturally, the reduced characteristic polynomial also satisfies Properties ($\chi$1)--($\chi$3).

\subsubsection{Chow rings}

Let $\sM=(E,\cL)$ be a loopless matroid and denote the collection of proper flats of $\sM$ by $\cL^*=\cL\setminus\{\emptyset,E\}$. The \emph{matroid Chow ring} is defined by
\[
A^*(\sM)=\frac{\Z\big[X_F\;|\;F\in\cL^*\big]}{\I+\J},
\]
where
\[
\I=\big\langle X_{F_1}X_{F_2}\;|\;F_1\text{ and }F_2\text{ are incomparable} \big\rangle
\]
and
\[
\J=\bigg\langle\sum_{e\in F }X_F-\sum_{f\in F}X_F\;\Big|\;e,f\in E \bigg\rangle.
\]
We denote the generators of the matroid Chow ring by $D_F=[X_F]\in A^1(M)$. Notice that the Chow ring only depends on the lattice of flats, which implies that $A^*(\sM)=A^*(\underline\sM)$.

Matroid Chow rings were first defined by Feichtner and Yuzvinsky \cite{FY} (in the more general setting of atomic lattices). The presentation given by Feichtner and Yuzvinsky slightly differs from the one give above in that it includes an additional generator $D_E\in A^1(\sM)$ and an additional relation
\[
D_E=-\sum_{e\in F\atop F\in\cL^*}D_F,
\]
where $e$ is any element of $E$. An important result of Feichtner and Yuzvinsky is the derivation of an integral basis for $A^*(\sM)$, which we recall here.

\begin{theorem}{\cite[Corollary~1]{FY}}\label{thm:basis}
If $\sM$ is a loopless matroid, then a $\Z$-basis of $A^*(\sM)$ is given by all monomials of the form
\[
D_{F_1}^{d_1}\cdots D_{F_\ell}^{d_\ell}
\]
with $\emptyset=F_0\subsetneq F_1\subsetneq \cdots\subsetneq F_k\subseteq E$ and $d_i<\rk(F_i)-\rk(F_{i-1})$ for all $i=1,\dots,\ell$.
\end{theorem}

Suppose that $\rk(\sM)=r+1$. It follows from Theorem~\ref{thm:basis} that $A^k(\sM)=0$ for any $k>r$, and that $A^r(\sM)$ is one-dimensional, generated by $D_E^r$. In particular, we can define a linear isomorphism
\[
\deg: A^r(\sM)\rightarrow\Z
\]
by setting $\deg\big((-D_E)^r\big)=1$. The class $-D_E$ played a central role in the work of Adiprasito, Huh, and Katz, where it was denoted as $\alpha$. In particular, Proposition~5.8 of \cite{AHK} implies that
\[
D_{F_1}\cdots D_{F_r}=(-D_E)^r
\]
for any complete flag $\emptyset\subsetneq F_1\subsetneq\cdots\subsetneq F_r\subsetneq E$. In other words, given any class $\gamma\in A^r(\sM)$, we can compute $\deg(\gamma)$ as follows.
\begin{enumerate}
\item Use the relations in $\I$ and $\J$ to find a linear combination
\[
\gamma=\sum_{\F}a_\F(\gamma)D_\F
\]
where the sum is over complete flags $\F=(\emptyset \subsetneq F_1\subsetneq \dots\subsetneq F_r\subsetneq E)$, the coefficients $a_\F(\gamma)$ are integers, and $D_\F=D_{F_1}\cdots D_{F_r}$.
\item Compute
\[
\deg(\gamma)=\sum_\F a_\F(\gamma).
\]
\end{enumerate}
The aforementioned result of Adiprasito, Huh, and Katz implies that the sum of the coefficients in (2) is independent of the choice of linear combination in (1).

Finally, in closing this section, we note that for the specific matroid $\sM=([n],2^E)$, the matroid Chow ring specializes to the Chow ring of Losev-Manin space $A^*(\LM_n)$ and the matroid degree map is identified with the algebro-geometric degree map. This observation motivates extending tools from $A^*(\LM_n)$ to Chow rings of arbitrary matroids.

\subsection{Matroid psi classes}\label{subsec:matroidpsi}

Throughout this subsection, we let $\sM=(E,\cL)$ denote a loopless matroid of rank $r+1$. We begin by using the characterization of Lemma~\ref{lem:psilinearcombo} to introduce a generalization of psi classes to the matroid setting.

\begin{definition}\label{def:matroidpsiclass}
For any $F\in\cL$ and $e\in E$, define classes $\psi_F^{\pm}\in A^1(\sM)$ by
\[
\psi_F^-=\sum_{G\in\cL^*\atop e\in G}D_G-\sum_{G\in\cL^*\atop G\supseteq F}D_G\;\;\;\text{ and }\;\;\;\psi_F^+=\sum_{G\in\cL^*\atop e\notin G}D_G-\sum_{G\in\cL^*\atop G\subseteq F}D_G.
\]
In the special case that $F=\emptyset$ or $F=E$, define
\[
\psi_0=\psi_\emptyset^+=\sum_{G\in\cL^*\atop e\notin G}D_G\;\;\;\text{ and }\;\;\;\psi_\infty=\psi_{E}^-=\sum_{G\in\cL^*\atop e\in G}D_G.
\]
\end{definition}

Notice that $\psi_\infty=-D_E$, which, as we mentioned above, was also denoted as $\alpha$ in \cite{AHK}, and we mention that $\psi_0$ also appeared in \cite{AHK}, where it was denoted $\beta$. We already commented above on why the class $\psi_\infty$ is independent of the choice of $e\in E$, and this also implies that $\psi_F^-$ is independent of this choice for any flat $F$. It is a short exercise to verify that $\psi_F^+$ is also independent of the choice of $e\in E$.

Equipped with a general definition of matroid psi classes, we now aim to generalize the basic results from the setting of Losev-Manin spaces. We start with the following generalization of Lemma~\ref{lem:selfintersection}.

\begin{proposition}\label{prop:selfintersection}
For any $F\in\cL^*$, we have
\[
D_F^2= D_F(-\psi_F^--\psi_F^+) \in A^2(\sM).
\]
\end{proposition}

\begin{proof}
Choose $e\in E$ and write
\begin{align}
\nonumber D_F&=D_F+\sum_{e\in G} D_G-\sum_{e\in G}D_G\\
&=D_F+\sum_{G} D_G-\sum_{e\notin G}D_G-\sum_{e\in G}D_G.\label{eq.termsterms}
\end{align}
When we multiply the first two terms of \eqref{eq.termsterms} by $D_F$ and use the fact that $D_FD_G=0$ when $F$ and $G$ are incomparable (by definition of $\I$), we have
\[
D_F\Big(D_F+\sum_G D_G\Big)=D_F\Big(\sum_{G\supseteq F}D_G+\sum_{G\subseteq F}D_G\Big).
\]
Including the final two terms of \eqref{eq.termsterms}, we conclude that
\[
D_F^2=D_F\bigg(-\Big(\sum_{e\in G}D_G-\sum_{G\supseteq F}D_G\Big)-\Big(\sum_{e\notin G}D_G-\sum_{G\subseteq F}D_G\Big)\bigg)=D_F(-\psi_F^--\psi_F^+).\qedhere
\]
\end{proof}

Repeatedly applying Proposition~\ref{prop:selfintersection} results in the following corollary. 

\begin{corollary}\label{cor:selfintersectionmatroid}
If $F_1,\dots,F_k\in\cL^*$ are distinct proper flats and $d_1,\dots,d_k$ are positive integers, then
\[
D_{F_1}^{d_1}\cdots D_{F_k}^{d_k}=\begin{cases}
D_\F\displaystyle\prod_{i=1}^k(-\psi_{F_i}^--\psi_{F_i}^+)^{d_i-1} & \parbox[]{8cm}{if, after possibly relabeling, $F_1,\dots,F_k$ form a flag $\F=(\emptyset\subsetneq F_1\subsetneq\dots\subsetneq F_k\subsetneq E)$,}\\
0 & \text{if }F_i\text{ and }F_j\text{ are incomparable for some }i,j.
\end{cases}
\]
\end{corollary}

Now that we know how to multiply arbitrary products of generators in $A^*(\sM)$, it remains to compute the degree of the resulting expression. The first step is the next result---generalizing Lemma~\ref{lem:product}---which reduces the computation to degrees of monomials in $\psi_0$ and $\psi_\infty$.

\begin{proposition}\label{prop:degreeofpsi2}
If $\F=(\emptyset=F_0\subsetneq F_1\subsetneq\dots\subsetneq F_k\subsetneq F_{k+1}=E)$ is a flag of flats and $a_0^+,a_1^-,a_1^+,\dots,a_{k}^-,a_k^+,a_{k+1}^-$ are nonnegative integers, then
\[
\deg_\sM\Big(D_\F\prod_{i=0}^{k}(\psi_{F_{i}}^+)^{a_{i}^+}(\psi_{F_{i+1}}^-)^{a_{i+1}^-}\Big)=\prod_{i=0}^k\deg_{\sM[F_i,F_{i+1}]}\Big(\psi_0^{a_{i}^+}\psi_\infty^{a_{i+1}^-}\Big).
\]
\end{proposition}

\begin{proof}
For each $i=0,\dots,k$, define an algebra homomorphism from the polynomial ring $\Z[X_G\;|\;G\in\cL(F_i,F_{i+1})]$ to the matroid Chow ring of $\sM$ as follows:
\begin{align*}
\phi_i:\Z[X_G\;|\;G\in\cL(F_i,F_{i+1})]&\rightarrow A^*(\sM) \\
X_G&\mapsto D_{G\cup F_i}.
\end{align*}
Unfortunately, the ideal $\J$ is not in the kernel of $\phi_i$, so $\phi_i$ does not descend to a homomorphism from the Chow ring $A^*(\sM[F_i,F_{i+1}])$. Let us modify $\phi_i$ by multiplying by $D_\F$:
\begin{align*}
\widehat\phi_i:\Z[X_G\;|\;G\in\F(F_i,F_{i+1})]&\rightarrow A^*(\sM) \\
\gamma&\mapsto \D_\F \phi_i(\gamma).
\end{align*}
Notice that $\widehat\phi_i$ is linear, but not multiplicative. We claim that the linear map $\widehat\phi_i$ descends to the Chow ring $A^*(\sM[F_i,F_{i+1}])$. To prove this, it suffices to check that the generators of both $\I$ and $\J$ are contained in the kernel of $\widehat\phi_i$.

First, notice that if $\emptyset\subsetneq G_1,G_2\subsetneq F_{i+1}\setminus F_{i}$ are incomparable, then $G_1\cup F_{i}$ and $G_2\cup F_{i}$ are also incomparable. This implies that $\widehat\phi_i(X_{G_1}X_{G_2})=0$ for incomparable $G_1,G_2\in\cL(F_i,F_{i+1})$, proving that $\hat\phi_i$ descends to the quotient by $\I$. Secondly, if $e,f\in F_{i+1}\setminus F_{i}$, then
\begin{align*}
\widehat\phi_i\Big(\sum_{G\in\cL(F_i,F_{i+1}) \atop e\in G}X_{G}-\sum_{G\in\cL(F_i,F_{i+1}) \atop f\in G}X_{G}\Big)
&=D_\F\Big(\sum_{F_{i}\subsetneq H\subsetneq F_{i+1} \atop e\in H}D_H-\sum_{F_{i}\subsetneq H\subsetneq F_{i+1} \atop f\in H}D_H\Big)\\
&=D_\F\Big(\sum_{ H\in\cL^* \atop e\in H}D_H-\sum_{H\in\cL^*\atop f\in H}D_H\Big)=0.
\end{align*}
The second equality above uses the following observations.
\begin{enumerate}
\item The only flats $H\in\cL^*$ that survive multiplication by $D_\F$ are those that are comparable with both $F_i$ and $F_{i+1}$.
\item If $e\in H$ or $f\in H$, then the only way that $H$ is comparable with $F_{i}$ is if $H\supsetneq F_{i}$.
\item If $H\supseteq F_{i+1}$, then $e,f\in H$, so the terms cancel in the difference in the final formula.
\end{enumerate}
Thus, $\widehat\phi_i$ descends to the quotient by $\I+\J$, and by a slight abuse of notation, we use the same notation to represent the induced linear map: $\widehat\phi_i: A^*(\sM[F_i,F_{i+1}])\rightarrow A^*(\sM)$.

Using multi-linearity, we combine the linear maps $\widehat\phi_i$ to obtain a linear map
\begin{align*}
\phi_\F:\bigotimes_{i=0}^{k}A^*(\sM[F_i,F_{i+1}])&\rightarrow A^*(\sM)\\
\gamma_0\otimes\dots\otimes \gamma_k&\mapsto D_\F\prod_{i=0}^k\phi_i(\gamma_i)
\end{align*}
Notice that, for any $e\in F_{i+1}\setminus F_{i}$, we have
\begin{align*}
\widehat\phi_i(\psi_0)=\widehat\phi_i\Big(\sum_{G\in\cL(F_i,F_{i+1}) \atop e\notin G}D_G\Big)
&=D_\F\sum_{F_{i}\subsetneq H\subsetneq F_{i+1} \atop e\notin H}D_H\\
&=D_\F\Big(\sum_{e\notin H}D_H-\sum_{H\subseteq F_{i}}D_H\Big)=D_\F\psi_{F_{i}}^+.
\end{align*}
Similarly, it can be checked that $\widehat\phi_i(\psi_\infty)=D_{\F}\psi_{F_{i+1}}^-$. It then follows from the definition of $\phi_\F$ that
\[
\phi_\F\Big(\bigotimes_{i=0}^{k}\psi_0^{a_{i}^+}\psi_\infty^{a_{i+1}^-}\Big)=D_\F\prod_{i=0}^{k}(\psi_{F_{i}}^+)^{a_{i}^+}(\psi_{F_{i+1}}^-)^{a_{i+1}^-}.
\]

Notice that, to compute $\deg_{\sM[F_i,F_{i+1}]}\big(\psi_0^{a_{i}^+}\psi_\infty^{a_{i+1}^-}\big)$, we can use the relations in $\I$ and $\J$ to find an express $\psi_0^{a_{i}^+}\psi_\infty^{a_{i+1}^-}$ as a linear combination of the form
\begin{equation}\label{eq:linearcombo}
\psi_0^{a_{i}^+}\psi_\infty^{a_{i+1}^-}=\sum_{\text{complete flags }\F^{(i)} \atop \text{in }\sM[F_i,F_{i+1}]}a_{\F^{(i)}} D_{\F^{(i)}},
\end{equation}
and then compute
\[
\deg_{\sM[F_i,F_{i+1}]}(\psi_0^{a_{i}^+}\psi_\infty^{a_{i+1}^-})=\sum_{\F^{(i)}}a_{\F^{(i)}}.
\]
Making one choice of expression \eqref{eq:linearcombo} for each $i=0,\dots,k$, we can apply $\phi_\F$ to obtain
\begin{align}\label{eq:sumoverflags}
\nonumber \phi_\F\Big(\bigotimes_{i=0}^{k}\psi_0^{a_{i}^+}\psi_\infty^{a_{i+1}^-}\Big)
&=\phi_\F\Big(\bigotimes_{i=0}^{k}\sum_{\F^{(i)}}a_{\F^{(i)}}D_{\F^{(i)}}\Big)\\
&=D_\F \sum_{\F^{(0)},\dots,\F^{(k)}}a_{\F^{(0)}}\cdots a_{\F^{(k)}}D_{\F^{(1)}\cup F_0}\cdots D_{\F^{(k)}\cup F_k},
\end{align}
where, for any flag $\F^{(i)}=(F_1^{(i)}\subsetneq\dots\subsetneq F_{k_i}^{(i)})$ of flats in $\sM[F_i,F_{i+1}]$, we define
\[
\F^{(i)}\cup F_{i}=(F_1^{(i)}\cup F_{i}\subsetneq\dots\subsetneq F_{k_i}^{(i)}\cup F_{i}),
\]
which is a flag of flats in $\sM$. Since each $\F^{(i)}$ is a complete flag of flats in $\sM[F_i,F_{i+1}]$, it follows that the sets in 
\[
\F\cup \bigcup_{i=0}^{k} (\F^{(i)}\cup F_i)
\]
form a complete flag of flats in $\sM$. Therefore, the products of generators in \eqref{eq:sumoverflags} are indexed by complete flags in $\sM$, and we conclude that
\begin{align*}
\deg_\sM\bigg(D_\F\prod_{i=0}^{k}(\psi_{F_{i}}^+)^{a_{i}^+}(\psi_{F_{i+1}}^-)^{a_{i+1}^-}\bigg)&=\sum_{\F^{(0)},\dots,\F^{(k)}}a_{\F^{(0)}}\cdots a_{\F^{(k)}}\\
&=\prod_{i=0}^{k}\sum_{\F^{(i)}}a_{\F^{(i)}}\\
&=\prod_{i=0}^{k}\deg_{\sM[F_i,F_{i+1}]}(\psi_0^{a_{i}^+}\psi_\infty^{a_{i+1}^-}).\qedhere
\end{align*}
\end{proof}

It now remains to compute degrees of monomials in $\psi_0$ and $\psi_\infty$. The key result in this regard---which is listed as Proposition~\ref{prop:degreeofpsi} below---relates these degree computations to the coefficients of reduced characteristic polynomials. This result was previously proved by Adiprasito, Huh, and Katz \cite[Proposition~9.5]{AHK}, but we find it instructive to give an alternative proof, motivated by the proof of Lemma~\ref{lem:degreeofpsi}, which uses properties of psi classes. We begin by introducing an analogue of the pullbacks of the forgetful maps.

\begin{proposition}\label{prop:pullbackpsi}
If $S\subseteq E$ is any subset, then there is a well-defined homomorphism $\rho_S:A^*(\sM|_S)\rightarrow A^*(\sM)$ defined on generators by
\[
\rho_S(D_G)=\sum_{G'\in\cL^* \atop G\subseteq G'\subseteq G\cup S^c}D_{G'}.
\]
In addition, if $S_1\subseteq S_2$ are nonempty flats, then $\rho_{S_1}=\rho_{S_1}\circ \rho_{S_2}$.
\end{proposition}

Notice that the sum in the definition of $\rho_S$ is over all flats of $\sM$ that are obtained from $G\subsetneq S$ by adding elements of $S^c$. In particular, the set $G$ is determined by any of the $G'$ via $G=G'\cap S$. In the special case of $\sM=([n],2^{[n]})$,we have $\rho_S=r_S^*$.

\begin{proof}[Proof of Proposition~\ref{prop:pullbackpsi}]
Define the homomorphism
\begin{align*}
\rho_S:\Z[X_G\;|\;G\in\cL_{\sM|_S}^*]&\rightarrow A^*(\sM)\\
X_G&\mapsto \sum_{G'\in\cL^* \atop G\subseteq G'\subseteq G\cup S^c}D_{G'}.
\end{align*}
To show that $\rho_S$ descends to a homomorphism from the Chow ring, we must verify that $\I$ and $\J$ are contained in the kernel of $\rho_S$. First, suppose that $G_1$ and $G_2$ are incomparable flats of $\sM|_S$. Then there exists $e,f\in S$ such that $e\in G_1\setminus G_2$ and $f\in G_2\setminus G_1$. Notice that every term in the sum defining $\rho_S(X_{G_1})$ is indexed by a set that contains $e$ but not $f$ and every term in the sum defining $\rho_S(X_{G_2})$ is indexed by a set that contains $f$ but not $e$. It follows that
\[
\rho_S(X_{G_1}X_{G_2})=\rho_S(X_{G_1})\rho_S(X_{G_2})=0\in A^*(\sM),
\]
showing that $\rho_S$ descends to the quotient by $\I$. Next, to show that $\rho_S$ descends to the quotient by $\J$, suppose that $e,f\in S$. Then
\begin{align*}
\rho_S\Big(\sum_{e\in G}X_G-\sum_{f\in G}X_G\Big)&=\sum_{e\in G\subseteq G'\subseteq G\cup S^c}D_{G'}-\sum_{ f\in G\subseteq G'\subseteq G\cup S^c}D_{G'}\\
&=\sum_{e\in G'\atop G'\not\supseteq S }D_{G'}-\sum_{f\in G'\atop G'\not\supseteq S}D_{G'}\\
&=\sum_{e\in G'}D_{G'}-\sum_{f\in G'}D_{G'}=0.
\end{align*}
The second equality above is implied by the following observations.
\begin{enumerate}
\item We are assuming that $\emptyset\subsetneq G\subsetneq S$. Since $G=G'\cap S$, this condition is equivalent to $\emptyset\subsetneq G'\cap S\subsetneq S$. Since $G'\cap S$ is nonempty (it always contains $e$ in the first sum and $f$ in the second), this is equivalent to $G'\not\supseteq S$.
\item Since $G=G'\cap S$ and $e\in S$, it follows that
\[
e\in G\Longleftrightarrow e\in (G'\cap S)\Longleftrightarrow e\in G'.
\]
\end{enumerate}
The third equality above is implied by the fact that every set $G'\supseteq S$ appears in both sums in the final expression, so these terms cancel. This completes the proof that $\rho_S$ descends to the quotient by $\J$. Thus, $\phi_S$ descends to the quotient by $\I+\J$ and induces the homomorphism whose existence is asserted in the proposition.

To finish the proof of the proposition, it remains to check that $\rho_{S_1}=\rho_{S_1}\circ \rho_{S_2}$. Notice that
\[
\rho_{S_1}\circ \rho_{S_2}(D_G)=\sum_{G\subseteq G'\subseteq G\cup S_2^c \atop G'\subseteq G''\subseteq G'\cup S_1^c}D_{G''}=\sum_{G\subseteq G''\subseteq G\cup S_1^c}D_{G''}=\rho_{S_1}(D_G),
\]
where the second equality uses that $G'$ is uniquely determined from $G''$ via $G'=G''\cap S_1$.
\end{proof}

The next result describes how $\psi_0$ and $\psi_\infty$ transform under the homomorphisms $\rho_S$ described in Proposition~\ref{prop:pullbackpsi}. The second statement of the result gives an alternative characterization of $\psi_F^{\pm}$ that generalizing Definition~\ref{def:generalpsi}.

\begin{proposition}\label{prop:pullbackpsi}
For any subset $S\subseteq E$, we have
\[
\rho_S(\psi_0)=\psi_0-\sum_{G\subseteq S^c} D_{G}\;\;\;\text{ and }\;\;\;\rho_S(\psi_\infty)=\psi_\infty-\sum_{G\supseteq S} D_{G}.
\]
In particular, if $F\in\cL^*$ is a proper flat, then
\[
\psi_F^-=\rho_F(\psi_\infty)\;\;\;\text{ and }\;\;\;\psi_F^+=\rho_{F^c}(\psi_0).
\]
\end{proposition}

\begin{proof}
For $\psi_0$, we compute
\begin{align*}
\rho_S(\psi_0)&=\rho_S\Big(\sum_{e\notin G}D_G\Big)\\
&=\sum_{e\notin G\subseteq G'\subseteq G\cup S^c}D_{G'}
\end{align*}
Arguing as in the proof of the previous proposition, the index in the last sum can be replaced with $e\notin G'$ and $G'\not\subseteq S^c$, proving that
\[
\rho_S(\psi_0)=\sum_{e\notin G'}D_{G'}-\sum_{G'\subseteq S^c} D_{G'}.
\]
The argument for $\rho_S(\psi_\infty)$ is similar.
\end{proof}

We now come to the generalization of Lemma~\ref{lem:degreeofpsi} to the matroid setting. As mentioned above, this result was previously proved by Huh and Katz \cite[Proposition~5.2]{HuhKatz}, though our formulation is more closely aligned with the presentation of Adiprasito, Huh, and Katz \cite[Proposition 9.5]{AHK}. Our proof relies on the recursive nature of the characteristic polynomial, and we note that this proof technique, using the deletion-contraction recursion, also appears in a different, more general form in recent work of Berget, Eur, Spink, and Tseng \cite[Theorem A]{BEST}.

\begin{proposition}{\cite[Proposition 9.5]{AHK}}\label{prop:degreeofpsi}
For nonnegative integers $a,b$, we have
\[
\deg_\sM(\psi_0^a\psi_\infty^b)=\begin{cases}
\mu^a(\sM) &\text{if }a+b=r,\\
0 &\text{else,}
\end{cases}
\]
where $\mu^a(M)$ is the $a$-th unsigned coefficient of the reduced characteristic polynomial of $\sM$:
\[
\overline\chi_\sM(\lambda)=\sum_{a=0}^r(-1)^a\mu^a(\sM)\lambda^{r-a}.
\]
\end{proposition}

Before proving Proposition~\ref{prop:degreeofpsi}, we briefly justify that it does, indeed, generalize Lemma~\ref{lem:degreeofpsi}. Suppose that $\sM=([n],2^{[n]})$ so that $A^*(\sM)=A^*(\LM_n)$. Then, for any subset $S\subseteq [n]$, we have $\rk(S)=|S|$, and it follows that
\[
\chi_\sM(\lambda)=\sum_{S\subseteq [n]}(-1)^{|S|}\lambda^{n-|S|}=\sum_{k=0}^n (-1)^k {n\choose k}\lambda^{n-k}=(\lambda-1)^n.
\]
Therefore,
\[
\overline\chi_\sM(\lambda)=(\lambda-1)^{n-1}
\]
and we conclude that $\mu^a(\sM)={n-1\choose a}$, as expected.

\begin{proof}[Proof of Proposition~\ref{prop:degreeofpsi}]
We prove the proposition by induction on $|E|$. If $|E|=1$, then $\ochi_\sM(\lambda)=1$ and $\mu^0(\sM)=1$, so the base case follows from the fact that $A^*(\sM)=A^0(\sM)=\Z$ and $\deg(\psi_0^0\psi_\infty^0)=1$.

We now turn to the induction step. Since $A^*(\sM)=A^*(\underline\sM)$, it suffices to assume throughout the induction step that $\sM$ is simple. First suppose that $e\in E$ is not a coloop. This implies that $\{e\}^c$ is not contained in any proper flats of $\sM$, and it then follows from Proposition~\ref{prop:pullbackpsi} that
\[
\rho_{\sM\setminus\{e\}}(\psi_\infty)=\psi_\infty.
\]
In particular, using that $\rk(\sM\setminus\{e\})=\rk(\sM)=r+1$ and that the degree map is determined by $\deg(\psi_\infty^r)=1$, this implies that
\[
\deg(\rho_{\sM\setminus\{e\}}(\gamma))=\deg(\gamma)\;\;\;\text{ for any }\;\;\;\gamma\in A^*(\sM\setminus\{e\}).
\]
Using our assumption that $\sM$ is simple, we have that $\{e\}\in\cL$, and it then follows from Proposition~\ref{prop:pullbackpsi} that
\[
\rho_{\sM\setminus\{e\}}(\psi_0)=\psi_0-D_{\{e\}}.
\]
Therefore, if $a+b=r$, then
\begin{align*}
\deg_{\sM\setminus\{e\}}(\psi_0^a\psi_\infty^b)&=\deg_\sM\big((\psi_0-D_{\{e\}})^a\psi_\infty^b\big)\\
&=\deg_\sM(\psi_0^a\psi_\infty^b)+(-1)^a\deg_{\sM}(D_{\{e\}}^a\psi_\infty^b)\\
&=\deg_\sM(\psi_0^a\psi_\infty^b)-\deg_{\sM/\{e\}}(\psi_0^{a-1}\psi_\infty^b),
\end{align*}
where the second equality follows from noting that we can write $\psi_0=\sum_{e\notin F}D_F$, in which case it follows $\psi_0D_{\{e\}}=0$, and the third equality follows from Proposition~\ref{prop:degreeofpsi2}. In the case where $a=0$, the second term in the final expression is equal to zero. The induction hypothesis then implies that
\[
\deg_\sM(\psi_0^a\psi_\infty^b)=\mu^a(\sM\setminus\{e\})+\mu^{a-1}(\sM/\{e\})=\mu^a(\sM),
\]
where the final equality is an application of Property ($\chi$3) for $\ochi_\sM(\lambda)$.

Next, suppose that $e\in E$ is a coloop. Since $e$ is a coloop and $\sM$ is simple, both $\{e\}$ and $\{e\}^c$ are flats of $\sM$. It follows from Proposition~\ref{prop:pullbackpsi} that
\[
\rho_{\sM\setminus\{e\}}(\psi_0)=\psi_0-D_{\{e\}}\;\;\;\text{ and }\;\;\;\rho_{\sM\setminus\{e\}}(\psi_\infty)=\psi_\infty-D_{\{e\}^c}.
\]
For any positive integer $a$, we have
\begin{align*}
\psi_0^a&=(\rho_{\sM\setminus\{e\}}(\psi_0)+D_{\{e\}})^a\\
&=\rho_{\sM\setminus\{e\}}(\psi_0^a)+\sum_{k=1}^a{a\choose k}D_{\{e\}}^k(-D_{\{e\}})^{a-k}\\
&=\rho_{\sM\setminus\{e\}}(\psi_0^a)-(-D_{\{e\}})^a,
\end{align*}
where the second equality uses the fact that $\psi_0D_{\{e\}}=0$ and the third equality uses that $\sum_{k=1}^a{a\choose k}(-1)^{a-k}=(-1)^{a-1}$. Similarly,
\[
\psi_\infty^b=\rho_{\sM\setminus\{e\}}(\psi_0^b)-(-D_{\{e\}^c})^b.
\]
Thus, if $a+b=r$, we compute that
\begin{align*}
\psi_0^a\psi_\infty^b&=\rho_{\sM\setminus\{e\}}(\psi_0^a\psi_\infty^b)-\rho_{\sM\setminus\{e\}}(\psi_0^a)(-D_{\{e\}^c})^b-(-D_{\{e\}})^a\rho_{\sM\setminus\{e\}}(\psi_\infty^b)+(-D_{\{e\}})^a(-D_{\{e\}^c})^b\\
&=-\rho_{\sM\setminus\{e\}}(\psi_0^a)(-D_{\{e\}^c})^b-(-D_{\{e\}})^a\rho_{\sM\setminus\{e\}}(\psi_\infty^b)\\
&=D_{\{e\}^c}\psi_0^a\psi_{\{e\}^c}^{b-1}+D_{\{e\}}\psi_{\{e\}}^{a-1}\psi_\infty^b
\end{align*}
where the second equality follows from observing that $\rk(\sM\setminus\{e\})=r$ and $D_{\{e\}}D_{\{e\}^c}=0$, and the third equality follows from the facts that $D_{\{e\}}\psi_0=D_{\{e\}^c}\psi_\infty=0$ and Corollary~\ref{cor:selfintersectionmatroid}. As before, terms with negative exponents are equal to zero. Computing degrees via Proposition~\ref{prop:degreeofpsi2}, we then see that
\[
\deg_\sM(\psi_0^a\psi_\infty^b)=\deg_{\sM\setminus\{e\}}(\psi_0^a\psi_\infty^{b-1})+\deg_{\sM/\{e\}}(\psi_0^{a-1}\psi_\infty^b).
\]
Since $e$ is a coloop, it follows that $\sM\setminus\{e\}=\sM/\{e\}$, because every flat not containing $e$ remains a flat when you add $e$ to it. Therefore, applying the induction hypothesis, we have
\[
\deg_\sM(\psi_0^a\psi_\infty^b)=\mu^a(\sM\setminus\{e\})+\mu^{a-1}(\sM\setminus\{e\})=\mu^a(\sM),
\]
where the final equality is an application of Property ($\chi$2) for $\ochi_\sM(\lambda)$. This completes the induction step, and finishes the proof.
\end{proof}

\subsection{Volume polynomials}

In this subsection, we illustrate the utility of psi classes by using them to reprove the main result in \cite{Eur} and one of the main results in \cite{BES}, both of which give an explicit formula for the volume polynomials of matroids. Given our parallel developments, the arguments in this setting are essentially verbatim generalizations of the arguments made in the setting of generalized permutahedra and Losev-Manin spaces.

Let $\sM=(E,\cL)$ be a loopless matroid of rank $r+1$. The \emph{volume polynomial} of $A^*(\sM)$ is the function
\begin{align*}
\Vol_\sM:A^1(\sM)&\rightarrow\Z\\
D&\mapsto\deg_\sM(D^r).
\end{align*}
Given a spanning set of generators $B=(B_1,\dots,B_m)$ for $A^1(\sM)$, the volume polynomial can be written explicitly as a homogeneous polynomial of degree $r$:
\[
\Vol_{\sM,B}(x_1,\dots,x_m)=\deg_\sM\bigg(\Big(\sum_{i=1}^mx_iB_i\Big)^{r}\bigg)\in\Z[x_1,\dots,x_m].
\]
In fact, given that $A^*(\sM)$ satisfies Poincar\'e duality (discussed in the next subsection), it follows from Lemma 13.4.7 in \cite{Toric} that the volume polynomial associated to any generating set determines a presentation for the Chow ring $A^*(\sM)$. Thus, it follows that the matroid Chow ring is determined from computations of the form
\[
\deg_\sM(B_1^{d_m}\cdots B_m^{d_m}),
\]
where $d_1+\dots+d_m=r$. The main result in \cite{Eur} is the computation of these degrees for the set of generators $(D_F\;|\;F\in\cL^*)$, and one of the main results in \cite{BES} is the computations of these degrees for the set of generators $(\psi_F^-\;|\; \emptyset\neq F\in\cL)$. We note that the authors of \cite{BES} stated their result in terms of classes that they denoted $h_F$, but it follows from the definitions that $h_F=\psi_F^-$. We now recover both of these computations using properties of psi classes.

The main result in \cite{Eur}, which implies Theorem~\ref{thm:eur}, is the following.

\begin{theorem}{\cite[Theorem 3.2]{Eur}}\label{thm:eur2}
If $\F=(\emptyset\subsetneq F_1\subsetneq\dots F_k\subsetneq E)$ is a flag of flats in $\sM$ and $d_1,\dots,d_k$ are positive integers that sum to $r$. Then
\[
\deg_\sM(D_{F_1}^{d_1}\cdots D_{F_k}^{d_k})=(-1)^{r-1}\prod_{i=1}^k{d_i-1\choose \tilde d_i-\rk(F_i)}\mu^{\tilde d_i-\rk(F_i)}(\sM[F_i,F_{i+1}]),
\]
with
\[
\tilde d_j=\sum_{i=1}^j d_i.
\] 
\end{theorem}

\begin{proof}
To prove this using psi classes, start by applying Corollary~\ref{cor:selfintersectionmatroid}:
\begin{align*}
D_1^{d_1}\cdots D_k^{d_k}&=D_\F\displaystyle\prod_{i=1}^k(-\psi_{F_i}^--\psi_{F_i}^+)^{d_i-1}\\
&=D_\F(-1)^{r-k}\sum_{a_i^-,a_i^+}\prod_{i=1}^k{d_i-1 \choose a_i^-,a_i^+}(\psi_{F_i}^-)^{a_i^-}(\psi_{F_i}^+)^{a_i^+}\\
&=D_\F(-1)^{r-k}\sum_{a_i^+=0}^{d_i-1}\prod_{i=1}^k{d_i-1 \choose a_i^+}(\psi_{F_i}^-)^{d_i-a_i^+-1}(\psi_{F_i}^+)^{a_i^+}\\
&=D_\F(-1)^{r-k}\sum_{a_i^+=0}^{d_i-1}\prod_{i=1}^{k}{d_i-1 \choose a_i^+}\prod_{i=0}^{k}(\psi_{F_{i}}^+)^{a_{i}^+}(\psi_{F_{i+1}}^-)^{d_{i+1}-a_{i+1}^+-1},
\end{align*}
where, by convention, we define $a_0^+=d_{k+1}-a_{k+1}^+-1=0$. By Proposition~\ref{prop:degreeofpsi}, the degree of each summand in this class is zero unless
\[
a_{i}^++d_{i+1}-a_{i+1}^+-1=\rk(F_{i+1})-\rk(F_{i})-1.
\]
These conditions have a unique solution with
\[
a_i^+=\tilde d_i-\rk(F_i).
\]
Thus, computing the degrees by Proposition~\ref{prop:degreeofpsi}, we have
\[
\deg_\sM(D_{F_1}^{d_1}\cdots D_{F_k}^{d_k})=(-1)^{r-k}\prod_{i=1}^k{d_i-1\choose \tilde d_i-\rk(F_i)}\prod_{i=0}^{k}\mu^{\tilde d_{i}-\rk(F_{i})}(\sM[F_i,F_{i+1}]).
\]
Theorem \ref{thm:eur2} follows by noting that $\mu^0(\sM)=1$ for any matroid $\sM$, so the $i=0$ term of the second product is $1$.
\end{proof}

One of the main results of \cite{BES}, which implies Theorem~\ref{thm:postnikov}, is the following.

\begin{theorem}{\cite[Theorem 5.2.4]{BES}}\label{thm:BES}
If $F_1,\dots,F_r$ are nonempty flats of $\sM$, then
\[
\deg_\sM(\psi_{F_1}^-\dots\psi_{F_r}^-)=
\begin{cases}
1 &\text{if } 0<i_1<\dots<i_k\leq r\Longrightarrow \rk(F_{i_1}\cup\dots\cup F_{i_k})>k,\\
0 &\text{else.}
\end{cases}
\]
\end{theorem}

\begin{proof}
To prove this result using properties of psi classes, first assume that there exists some $0<i_1<\dots<i_k\leq r$ such that $\rk(F_{i_1}\cup\dots\cup F_{i_k})\leq k$. Denote $S=F_{i_1}\cup\dots\cup F_{i_k}$. By Proposition~\ref{prop:pullbackpsi}, we compute that
\[
\psi_{F_{i_1}}^-\cdots\psi_{F_{i_k}}^-=\rho_{F_{i_1}}(\psi_\infty)\cdots \rho_{F_{i_k}}(\psi_\infty)=\rho_{S}\big(\rho_{F_{i_1}}(\psi_\infty)\cdots \rho_{F_{i_k}}(\psi_\infty)\big).
\]
The input of $\rho_{S}$ is a class in $A^k(\sM|_S)$, which is zero because 
\[
\rk(M|_S)=\rk(S)\leq k.
\]
Thus,
\[
\deg_\sM(\psi_{F_1}^-\dots\psi_{F_r}^-)=\deg(0)=0.
\]

Next, suppose that $0<i_1<\dots<i_k\leq r$ implies that $\rk(F_{i_1}\cup\dots\cup F_{i_k})>k$. By definition,
\begin{align*}
\psi_{F_1}^-\dots\psi_{F_{r}}^-&=\Big(\psi_\infty-\sum_{G\supseteq F_{1}}D_G\Big)\cdots\Big(\psi_\infty-\sum_{G\supseteq F_{{r}}}D_G\Big)\\
&=\sum_{k=0}^{r}\psi_\infty^{r-k}(-1)^k\sum_{0<i_1<\dots<i_k\leq r \atop G_j\supseteq F_{i_j}}D_{G_1}\cdots D_{G_k}.
\end{align*}
We claim that the only nonzero term in the sum is the one indexed by $k=0$. To see why, notice that multiplying $D_{G_1}\cdots D_{G_k}$ will either be zero if $G_i$ and $G_j$ are incomparable for some $i$ and $j$ or it will be a multiple of $D_\mathcal{G}$ for some flag $\mathcal{G}$. In the latter case, the largest flat in the flag $\mathcal{G}$ must be $G=G_1\cup\cdots\cup G_k$, which contains $F_{i_1}\cup\dots\cup F_{i_k}$. This implies that $\rk(G)>k$. It follows that
\[
\deg_\sM(\psi_\infty^{r-k}D_{G_1}\cdots D_{G_k})=0,
\]
because, when expanded using Proposition~\ref{prop:degreeofpsi2}, the exponent of $\psi_\infty$ appearing in the final term of the product is $r-k\geq\rk(E)-\rk(G)=\rk(\sM[G,E])$. Thus, the only nonzero term in the sum is
\[
\deg_\sM(\psi_\infty^r)=\mu^0(\sM)=1.\qedhere
\]
\end{proof}

\subsection{Poincar\'e duality}

In this final section, we describe one more application of our developments of psi classes, which is a new proof of the Poincar\'e duality property for matroid Chow rings. Our proof utilizes the following computational result.

\begin{lemma}\label{lem:poincare}
If $\F=(\emptyset\subsetneq F_1\subsetneq\cdots\subsetneq F_k\subsetneq E)$ is a flag of flats in $\sM$ and we have integers $d_1,\dots,d_k>0$ and $d_E\geq 0$ that sum to $r$, then
\begin{enumerate}
\item $\deg_\sM(D_{F_1}^{d_1}\cdots D_{F_k}^{d_k}D_E^{d_E})=0$ if $d_E+\displaystyle\sum_{i=m}^k d_i>r-\rk(F_{m-1})$ for some $m\in\{1,\dots,k\}$ and
\item $\deg_\sM(D_{F_1}^{d_1}\cdots D_{F_k}^{d_k})=(-1)^{r-k+1}$ if $d_E+\displaystyle\sum_{i=m}^k d_i=r-\rk(F_{m-1})$ for all $m\in\{1,\dots,k\}$.
\end{enumerate}
\end{lemma}

\begin{proof}
By Corollary~\ref{cor:selfintersectionmatroid}, we have
\begin{align*}
D_{F_1}^{d_1}\cdots D_{F_k}^{d_k}D_E^{d_E}&=D_\F(-\psi_\infty)^{d_E}\prod_{i=1}^k(-\psi_{F_i}^--\psi_{F_i}^+)^{d_i-1}\\
&=(-1)^{r-k+1}\sum_{a_i^+=0}^{d_i-1}\prod_{i=1}^k{d_i-1\choose a_i^+}D_\F\prod_{i=0}^k(\psi_{F_{i}}^+)^{a_{i}^+}(\psi_{F_{i+1}}^-)^{d_{i+1}-1-a_{i+1}^+}
\end{align*}
where $a_0^+=0$, $F_{k+1}=E$, and $a_{k+1}^-=d_\infty$. Computing the degree using Proposition~\ref{prop:degreeofpsi2}, we see that the degree is nonzero only if 
\[
a_i^++d_{i+1}-1-a_{i+1}^+=\rk(F_{i+1})-\rk(F_i)-1\;\;\;\text{ for all }\;\;\;i=0,\dots,k.
\]
The unique solution of this system is given by
\[
a_m^+=r-\rk(F_m)-d_E-\sum_{i=m+1}^\ell d_i\;\;\;\text{ for all }\;\;\;m=1,\dots,k.
\]
Property (1) follows from the observation that $a_m^+\geq 0$ for all $m=1,\dots,k$. Notice that the condition in Property (2) implies that $a_m^+=0$ for all $m=1,\dots,k$, and Property (2) then follows from Proposition~\ref{prop:degreeofpsi2} and the fact that
\[
\deg_{\sM[F_i,F_{i+1]}}\big(\psi_\infty^{\rk(F_{i+1})-\rk(F_i)-1}\big)=1.\qedhere
\]
\end{proof}

We now use the Feichtner--Yuzvinsky basis for $A^*(\sM)$ to prove Poincar\'e duality.

\begin{theorem}\label{thm:poincareduality}
For any $k\in 0,\dots,r$, the map 
\begin{align*}
\phi_k:A^k(\sM)&\rightarrow A^{r-k}(\sM)^\vee\\
\gamma&\mapsto (\mu\mapsto\deg_\sM(\mu\gamma))
\end{align*}
is an isomorphism of $\Z$-modules.
\end{theorem}

\begin{proof}
Recall that the Feichtner--Yuzvinsky basis for $A^k(\sM)$ comprises all monomials of the form
\[
B=D_{F_1}^{d_1}\cdots D_{F_\ell}^{d_\ell}
\]
where $\emptyset=F_0\subsetneq F_1\subsetneq\cdots\subsetneq F_\ell\subseteq E$ and $0<d_i<\rk(F_i)-\rk(F_{i-1})$ for all $i=1,\dots,\ell$ with $\sum_{i=1}^\ell d_i=k$. Throughout this proof, we always assume that $F_\ell=E$ while allowing for the possibility that $d_\ell=0$. For each such basis element $B$, define a corresponding basis element $\widehat B\in A^{r-k}(\sM)$ by
\[
\widehat B=D_{F_1}^{\hat d_1}\cdots D_{F_\ell}^{\hat d_\ell},
\]
where
\[
\hat d_i=\begin{cases}
\rk(F_i)-\rk(F_{i-1})-d_i &\text{if }i<\ell,\\
r-\rk(F_{\ell-1})-d_i &\text{if }i=\ell.
\end{cases}
\]
Let $\widehat B^\vee\in A^{r-k}(\sM)^\vee$ denote the dual of $\widehat B$. We can write $\phi_k$ as a square matrix whose rows are indexed by the basis elements $B$ and whose columns are indexed by the corresponding basis elements $\widehat B^\vee$. The $(B_1,\widehat B_2^\vee)$ entry of this matrix is $\deg_{\sM}(B_1\widehat B_2)$, which can be computed explicitly by the results of Subsection~\ref{subsec:matroidpsi}. To prove the statement in the theorem, we show that this matrix is invertible over $\Z$.

First, notice that the element $\widehat B$ was constructed so that 
\[
\sum_{i=m}^\ell (d_i+\hat d_i)=r-\rk(F_{m-1})\;\;\;\text{ for all }\;\;\;m\in\{1,\dots,\ell\}, 
\]
so Lemma~\ref{lem:poincare}(2) implies that $\deg_{\sM}(B\widehat B)=(-1)^{r-\ell+1}$. This implies that the diagonal entries of the matrix are all $\pm 1$. To finish the proof, it suffices to prove that the matrix is triangular with respect to some choice of ordering on the bases.

For each basis element $B=D_{F_1}^{d_1}\cdots D_{F_\ell}^{d_\ell}$ as above, define a multidegree by
\[
\delta(B)=(\hat d_\ell,\rk(F_{\ell-1}),\hat d_{\ell-1},\rk(F_{\ell-2}),\dots,\rk(F_1),\hat d_1,0,0,\dots).
\]
The multidegree defines a lexicographic partial ordering on the basis, and we let $\prec$ be any total ordering of the basis that refines the lexicographic partial ordering induced by $\delta$. In other words, we insist that $B\prec B'$ only if $\delta(B)$ is less than or equal to $\delta(B')$ in the lexicographic partial ordering. We claim that $\phi_k$ is lower triangular with respect to this order. To prove this, suppose that $B\prec B'$; we must prove that $\deg_{\sM}(B\widehat B')=0$.

First, consider the case where $\delta(B)=\delta(B')$. It follows that $\ell=\ell'$, and $\rk(F_i)=\rk(F_i')$ and $d_i=d_i'$ for all $i=1,\dots,\ell$. Since $B$ and $B'$ are not the same monomial, it must be the case that $F_i$ is incomparable to $F_i'$ for some $i$. Since $B$ has a factor of $F_i$ and $\widehat B'$ has a factor of $F_i'$, it follows that $B\widehat B'=0$, so $\deg_{\sM}(B\widehat B')=0$.

Next, consider the case where $\delta(B)\neq \delta(B')$. We first suppose that the first entry where they differ is $\rk(F_m)<\rk(F_m')$. This implies that $\hat d_i=\hat d_i'$ and $\rk(F_i)=\rk(F_i')$ for all $i>m$. Since $B$ has a nontrivial factor of $F_i$ and $\widehat B'$ has a nontrivial factor of $F_i'$, and these are flats of the same rank for $i>m$, the only way that $B\widehat B'\neq 0$ is if $F_i=F_i'$ for all $i>m$. Assuming that this is the case, we can write
\[
B\widehat B'=F_\ell^{e_\ell}F_{\ell-1}^{e_{\ell-1}}\cdots F_{m+1}^{e_{m+1}}F_{m'}^e\cdots
\]
where the tail of the product consists of powers of flats of lower rank. Notice that
\begin{align*}
\sum_{i=m+1}^\ell e_i&=\sum_{i=m+1}^\ell(d_i+\hat d_i')\\
&=\sum_{i=m+1}^\ell(d_i+\hat d_i)\\
&=r-\rk(F_m)\\
&>r-\rk(F_m'),
\end{align*}
from which Lemma~\ref{lem:poincare}(1) implies that $\deg_\sM(B\widehat B')=0$.

Lastly, suppose that the first entry where $\delta(B)$ and $\delta(B')$ differ is $\hat d_m<\hat d_m'$. This implies that $\hat d_i=\hat d_i'$ for all $i>m$ and $\rk(F_i)=\rk(F_i')$ for all $i\geq m$. As in the previous case, we can write
\[
B\widehat B'=F_\ell^{e_\ell}F_{\ell-1}^{e_\ell-1}\cdots F_{m}^{e_{m}}F^e\cdots,
\]
where $F$ is equal to the flat in $\{F_{m-1},F_{m-1}'\}$ with highest rank. We then compute
\begin{align*}
\sum_{i=m}^\ell e_i&=\sum_{i=m}^\ell(d_i+\hat d_i')\\
&>\sum_{i=m}^\ell(d_i+\hat d_i)\\
&=r-\rk(F_{m-1})\\
&\geq r-\rk(F)
\end{align*}
from which Lemma~\ref{lem:poincare}(1) implies that $\deg_\sM(B\widehat B')=0$, completing the proof.
\end{proof}

\bibliographystyle{alpha}
%\bibliography{references}

\newcommand{\etalchar}[1]{$^{#1}$}

\end{document}